\documentclass[a4paper,11pt]{amsart}
\usepackage[backend=bibtex,style=numeric,giveninits=true,isbn=false, url=false, doi=false]{biblatex}
\renewbibmacro*{issue+date}{%
  \ifboolexpr{not test {\iffieldundef{year}} or not test {\iffieldundef{issue}}}
    {\printtext[parens]{%
       \iffieldundef{issue}
         {\usebibmacro{date}}
         {\printfield{issue}%
          \setunit*{\addspace}%
          \usebibmacro{date}}}}
    {}%
  \newunit}
\renewbibmacro{in:}{}
 
\AtEveryBibitem{\clearfield{URL}\clearfield{DOI}}
\usepackage{amsmath}
\usepackage{amsthm,amsfonts,amssymb,graphics}
\usepackage{pgfplots}

\usepackage{subcaption}

\usepackage{hyperref}
\usepackage{color}
\usepackage[normalem]{ulem}
\usepackage{enumerate}
\usepackage{a4wide}
\usepackage{multirow}
\usepackage{booktabs}
\usepackage[foot]{amsaddr}
\usetikzlibrary{calc}
\usepackage{dsfont}

\numberwithin{equation}{section}

\newcommand {\R} {\mathbb{R}}

\numberwithin{equation}{section}
\theoremstyle{plain}
\newtheorem{theorem}{Theorem}[section]
\newtheorem{lemma}[theorem]{Lemma}
\newtheorem{proposition}[theorem]{Proposition}
\newtheorem{corollary}[theorem]{Corollary}

\theoremstyle{remark}
\newtheorem{remark}[theorem]{Remark}
\newtheorem{definition}[theorem]{Definition}
\newtheorem{assumption}[theorem]{Assumption}
\newtheorem{question}[theorem]{Question}
\newtheorem{conjecture}[theorem]{Conjecture}

\mathchardef\mhyphen="2D

\begin{document}

\title[Fluctuations of the number of excursion sets ]{Fluctuations of the number of excursion sets  of planar Gaussian fields } 
\author{Dmitry Beliaev\textsuperscript{1} }
\address{\textsuperscript{1}Mathematical Institute, University of Oxford}
\email{belyaev@maths.ox.ac.uk}

\author{Michael McAuley\textsuperscript{1,2} }
\address{\textsuperscript{2}Present address:  Department of Mathematics and Statistics, University of Helsinki.} 
\email{michael.mcauley@helsinki.fi}

\author{Stephen Muirhead \textsuperscript{3,4}}
\address{ \textsuperscript{3}School of Mathematical Sciences, Queen Mary University of London}
\address{\textsuperscript{4}Present address: School of Mathematics and Statistics, University of Melbourne}

\email{smui@unimelb.edu.au}

		

		\begin{abstract}
For a smooth, stationary, planar Gaussian field, we consider the number of connected components of its excursion set (or level set) contained in a large square of area $R^2$. The mean number of components is known to be of order $R^2$ for generic fields and all levels. We show that for certain fields with positive spectral density near the origin (including the Bargmann-Fock field), and for certain levels $\ell$, these random variables have fluctuations of order at least $R$, and hence variance of order at least $R^2$. In particular this holds for excursion sets when $\ell$ is in some neighbourhood of zero, and it holds for excursion/level sets when $\ell$ is sufficiently large. We prove stronger fluctuation lower bounds of order $R^\alpha$, $\alpha \in [1,2]$, in the case that the spectral density has a singularity at the origin. Finally we show that the number of excursion/level sets for the Random Plane Wave at certain levels has fluctuations of order at least $R^{3/2}$, and hence variance of order at least~$R^3$. We expect that these bounds are of the correct order, at least for generic levels.
		\end{abstract}
		
		\maketitle

	\section{Introduction}\label{s:introduction}
	Let $f:\R^2\to\R$ be a continuous centred stationary Gaussian field. In this paper we study the (upper-)excursion sets and level sets of $f$, that is, the random sets
	\begin{displaymath}
	\{f\geq\ell\}:=\left\{x\in\R^2 \middle|  f(x)\geq\ell\right\}\quad\text{and}\quad\{f=\ell\}:=\left\{x\in\R^2\middle|f(x)=\ell\right\}
	\end{displaymath}
	for $\ell\in\R$. For a wide class of fields, and for many levels $\ell$, we derive lower bounds on the fluctuations of the number of connected components of these sets contained inside large domains. We expect that these bounds are of the correct order, at least for generic levels.
	
	Gaussian fields are used as a model for spatial phenomena in many fields of science (e.g.\ in quantum chaos \cite{jain2017nodal}, medical imaging \cite{worsley1996unified}, oceanography \cite{azais2009level}, cosmology \cite{bardeen86} etc.), and the analysis of their excursion/level sets has many potential applications. To give an example, cosmological theories predict that the Cosmic Microwave Background Radiation (CMBR) can be modelled as a realisation of an isotropic Gaussian field on the two-dimensional sphere~\cite{bardeen86}. One way to test this prediction is to compare geometric properties of the excursion/level sets of the CMBR with the Gaussian model; for instance, a recent analysis \cite{pranav2018unexpected} used the number of excursion set components as a test statistic. We expect that a rigorous understanding of the fluctuations of this quantity will make such statistical analyses more robust.
	
	The number of connected components of the excursion/level sets of a Gaussian field are inherently difficult quantities to study because they are `non-local'; the number of components in a domain cannot be counted by partitioning the domain and summing the number of components in each sub-domain, since some components will intersect multiple sub-domains. This can be contrasted with other `local' functionals, such as the length of a level set, the number of critical points, or the Euler characteristic of an excursion set (the locality of which can be seen from the Gauss--Bonnet theorem).
	
	Functionals of Gaussian fields which are `non-local' cannot easily be analysed using classical tools such as the Kac-Rice formula \cite[Chapter 11]{RFG} or the Wiener chaos expansion \cite[Chapter 2]{jan}. Nevertheless, the number of excursion/level set components of planar Gaussian fields have recently been studied using other, more general, techniques. Nazarov and Sodin \cite{SodinNazarov2015asymptotic} used an ergodic argument to prove a law of large numbers. Specifically, they showed that if $f$ is an ergodic field satisfying some regularity assumptions, $D_R:=(-R/2,R/2)^2$ is the open square of side length $R$ centred at the origin, and $N_\mathrm{LS}(D_R,\ell)$ denotes the number of components of the level set $\{f=\ell\}$ contained in $D_R$ (i.e.\ those which intersect $D_R$ but not $\partial D_R$), then there exists a constant $c_\mathrm{LS}>0$ such that
	\begin{displaymath}
	\frac{N_\mathrm{LS}(D_R,0)}{R^2}\to c_\mathrm{LS}
	\end{displaymath}
	as $R\to\infty$, where convergence occurs in $L^1$ and almost surely. Although this result was stated only for the nodal set (i.e.\ the zero level set), the arguments in \cite{SodinNazarov2015asymptotic} go through verbatim for excursion/level sets at arbitrary levels.
	
	Results on the fluctuations of the number of excursion/level set components are comparatively lacking. Each excursion set component contains at least one critical point, and each level set component `surrounds' an (upper or lower) excursion set component. Since the number of critical points in a domain has a finite second moment which scales like the square of the area of the domain \cite{eliz85,esfo16}, it follows that there exists a positive constant $c_1=c_1(\ell)$ such that, for all sufficiently large $R$,
	\begin{equation}
	\label{e:ub}
	\mathrm{Var}(N_\mathrm{LS}(D_R,\ell)) < c_1 R^4 \quad \text{and} \quad \mathrm{Var}(N_\mathrm{ES}(D_R,\ell)) < c_1 R^4,
	\end{equation}
where $N_\mathrm{ES}(D_R,\ell)$ denotes the number of components of $\{f\geq \ell\}$ contained in $D_R$. While the upper bound of order $R^4$ is attained in certain degenerate cases (see Proposition~\ref{t:Delta mass zero}), it is expected that the number of excursion/level sets of generic fields (i.e.\ those with rapid correlation decay) has variance of order exactly $R^2$ (see Section~\ref{ss:diss}).
	
	To the best of our knowledge, up until now the only non-trivial lower bound on the variance of either $N_\mathrm{LS}(D_R,\ell)$ or $N_\mathrm{ES}(D_R,\ell)$ is the recent result of Nazarov and Sodin \cite{ns19} that $\mathrm{Var}(N_\mathrm{LS}(D_R,0))$ grows at least like some positive power of $R$ (more precisely, they consider a related model of sequences of Gaussian fields on the sphere); the exponent in their bound is unspecified and not expected to be optimal. It is unclear whether their methods extend to studying $N_\mathrm{LS}(D_R,\ell)$ for $\ell \neq 0$ or to $N_\mathrm{ES}(D_R,\ell)$. Nazarov and Sodin \cite{nazarov2009number} have also improved the upper bound: they have shown that $\mathrm{Var}(N_\mathrm{LS}(D_R,0)) < c R^{4- 2/15}$ in the case of random spherical harmonics (which are closely related to the Random Plane Wave that we discuss below). Weaker concentration bounds have also been established for general fields \cite{rivera2017quasi, Beliaev2018Cov}.
	
	In this work, we prove lower bounds on $\mathrm{Var}(N_\mathrm{LS}(D_R,\ell))$ and $\mathrm{Var}(N_\mathrm{ES}(D_R,\ell))$ that are, conjecturally at least, of the correct order. To summarise our main results (see Theorems \ref{t:fluctuations general} and~\ref{t:fluctuations RPW}), we show that for a wide class of Gaussian fields there exists an exponent $\alpha \in [2,4]$ such that, for many levels $\ell$,
	\begin{equation}
	\label{e:main}
	\mathrm{Var}(N_\mathrm{LS}(D_R,\ell)) > c R^\alpha \quad \text{and} \quad \mathrm{Var}(N_\mathrm{ES}(D_R,\ell))> c R^\alpha .
	\end{equation}
	for some $c = c(\ell) > 0$ and all $R$ sufficiently large. The value of $\alpha \in [2,4]$ is explicit and depends on the behaviour at the origin of the spectral measure of the field (see \eqref{e:sm} for the definition of the spectral measure). For fields with rapid correlation decay and positive spectral density at the origin, the bound \eqref{e:main} holds for $\alpha=2$, whereas for fields whose spectral measure has a singularity at the origin, \eqref{e:main} holds for an $\alpha \in (2,4)$ that depends on the polynomial exponent of the singularity. We also study the important special case of the Random Plane Wave, for which we show that the bound \eqref{e:main} holds for $\alpha = 3$. Interestingly, this result is inconsistent with the predictions of the well-known Bogomolny-Schmit conjecture~\cite{bogomolny2002percolation} that $\mathrm{Var}(N_\mathrm{LS}(D_R,\ell)) \sim c R^2$ for the Random Plane Wave (see the discussion in Section~\ref{ss:diss}), although our results do not apply to the nodal set which is the most important case of the conjecture.
	
	We establish the variance bounds in \eqref{e:main} for a wide range of levels. For general fields, the bound for excursion sets holds for all levels $\ell$ in a neighbourhood of zero (the nodal level $\ell = 0$ is excluded for the Random Plane Wave), and when $\ell$ is sufficiently large the bound holds for both excursion and level sets (see Corollaries \ref{c:BF}, \ref{c:general} and \ref{c:Variance RPW}). Indeed, Theorems~\ref{t:fluctuations general} and~\ref{t:fluctuations RPW} give a sufficient condition for \eqref{e:main} which we expect to be satisfied for all but a very small, finite number of levels $\ell$. In fact we suspect (see Section~\ref{ss:diss}) that this condition should fail for only one or three values of $\ell$ (depending on the field).
	
	On the other hand, we do not expect that \eqref{e:main} is necessarily true for all levels. While we conjecture that \eqref{e:main} holds for generic levels, we expect that for some fields there exists a finite set of `anomalous' levels at which the variance is of strictly lower order (see Conjectures~\ref{c:1} and \ref{c:2} for a precise statement). This phenomenon is reminiscent of `Berry cancellation', i.e.\ the known fact that, for some fields such as the Random Plane Wave, the variance of the length of the nodal set is of lower order than for non-zero levels \cite{berry2002statistics, wigman2010fluctuations, npr}. 
	
	
	\section*{Acknowledgements}
	The first author was partially funded by the Engineering \& Physical Sciences Research Council (EPSRC) Fellowship EP/M002896/1 and partially supported by the Ministry of Science and Higher Education of Russia (subsidy in the form of a grant for creation and development of International Mathematical Centers, agreement no.  075-15-2019-1620, November 8, 2019). The second author was supported by the European Research Council (ERC) Advanced Grant QFPROBA (grant number 741487). The third author was partially supported by the Australian Research Council (ARC) Discovery Early Career Researcher Award DE200101467.
	
	The authors would like to thank an anonymous referee for many comments which substantially improved the arrangement of our arguments and especially for pointing out an error in the original proof of Lemma~\ref{l:RPW perturbation}. The authors would also like to thank Benedetta Cavalli for making us aware of a mistake in the initial proof of Theorem~\ref{t:fluctuations RPW}.

	\section{Main results}
	We consider a Gaussian field ${f:\R^2\to\R}$ which is continuous, centred and stationary, and let $\kappa( x):=\mathbb{E}(f( x)f( 0))$ be its covariance function. Throughout the paper we make the following basic assumption:
	
	\begin{assumption}\label{a:minimal}
		The Gaussian field $f:\R^2\to\R$ is $C^3$-smooth almost surely, and normalised so that, for each $x \in \mathbb{R}^2$,
		\begin{equation}
		\label{e:norm}
		\mathbb{E}(f( x)) = 0 \, ,  \quad \mathrm{Var}(f( x)) = 1 \quad \text{and} \quad \mathrm{Cov}( \nabla f( x) ) = c I_2,  
		\end{equation}
		where $c$ is a positive constant and $I_2$ is the $2\times 2$ identity matrix. In addition we assume that
		\[ \max_{\lvert\alpha\rvert\leq 2}\left\lvert\partial_\alpha\kappa( x)\right\rvert\to 0 \qquad \text{as } \lvert  x\rvert\to\infty , \]
		and that, for every $x\in\R^2\backslash\{0\}$,
		\begin{equation}
		\label{e:nondegen}
		\nabla^2 f(0) \quad \text{and} \quad (f(x),f(0),\nabla f(x),\nabla f(0))  
		\end{equation}
		are non-degenerate Gaussian random variables.
	\end{assumption}
	
	This assumption implies, in particular, that $\kappa$ is of class $C^6$ \cite[Appendix~A.3]{SodinNazarov2015asymptotic}, and also that the field is ergodic \cite[Appendix~B]{SodinNazarov2015asymptotic}. We impose the normalisation \eqref{e:norm} for simplicity; since $(f(0),\nabla f(0))$ is assumed to be non-degenerate we can always apply a linear rescaling and rotation to the domain of $f$ so that \eqref{e:norm} holds. A sufficient condition for \eqref{e:nondegen} to be non-degenerate is that the support of the spectral measure (see the definition in \eqref{e:sm}) contains either an open set or a centred ellipse \cite[Lemma~A.2]{Beliaev2019smoothness}.
	
	We begin by formally stating the law of large numbers for excursion/level sets (noting that this actually holds under weaker conditions than those which we give). We fix an open rectangle $D\subset\R^2$ centred at the origin. For $R\geq 1$, we let $D_R=\{x\in\R^2:x/R\in D\}$ and let 	$N_\mathrm{ES}(D_R,\ell)$ denote the number of components of $\{f\geq\ell\}$ contained in $D_R$ (i.e.\ those which intersect $D_R$ but not $\partial D_R$). We define $N_\mathrm{LS}(D_R,\ell)$ analogously for $\{f=\ell\}$.
	
	\begin{theorem}[\cite{SodinNazarov2015asymptotic,kurlberg2017variation,Beliaev2018Number}]\label{t:main level}
		Let $f$ be a Gaussian field satisfying Assumption \ref{a:minimal}. For each $\ell\in\R$, there exist $c_\mathrm{ES}(\ell),c_\mathrm{LS}(\ell)\geq 0$ such that
\[
\begin{aligned}
\mathbb{E}(N_\mathrm{ES}(D_R,\ell))=c_\mathrm{ES}(\ell)\cdot \mathrm{Area}(D)\cdot R^2+O(R),\\
\mathbb{E}(N_\mathrm{LS}(D_R,\ell))=c_\mathrm{LS}(\ell)\cdot \mathrm{Area}(D)\cdot R^2+O(R)
\end{aligned}
\]
as $R\to\infty$. The constants implied by the $O(\cdot)$ notation are independent of $\ell$. Furthermore
\begin{align*}
\frac{N_\mathrm{ES}(D_R,\ell)}{\mathrm{Area}(D)\cdot R^2} \rightarrow c_\mathrm{ES}(\ell) \quad \text{and} \quad 
\frac{N_\mathrm{LS}(D_R,\ell)}{\mathrm{Area}(D)\cdot R^2}\rightarrow c_\mathrm{LS}(\ell)
\end{align*}
almost surely and in $L^1$.
\end{theorem}
The limiting constants $c_\mathrm{ES}(\ell)$ and $c_\mathrm{LS}(\ell)$ describe the asymptotic density of excursion sets and level sets respectively. Since they are defined implicitly, very little is known rigorously about them. In \cite{Beliaev2018Number} a representation was given in terms of critical points of various types. For $R > 0$ and $a \le b$, we define $N_{h}(D_R,[a,b])$, for $h = m^+, m^-, s^+, s^-$, to be the number of local maxima, local minima, upper connected saddles and lower connected saddles respectively of $f$ in $D_R$ with level in $[a,b]$ (see \cite{Beliaev2018Number} for the definition of upper/lower connected saddles; the precise definition has no relevance to the current paper).
	
	\begin{theorem}[{\cite[Proposition~1.8, Theorem~1.9]{Beliaev2018Number}}]
		\label{t:integral equality}
		Let $f$ be a Gaussian field satisfying Assumption \ref{a:minimal} and $D$ be an open rectangle centred at the origin. Then for all $R > 0$ and $a \le b$,
\[ c_\mathrm{ES}(a) - c_\mathrm{ES}(b) = \frac{1}{\mathrm{Area}(D) R^2} \Big( \mathbb{E} ( N_{m^+}(D_R,[a,b]) - N_{s^-}(D_R,[a,b]) ) \Big) \]
		and
\begin{multline*}
		c_\mathrm{LS}(a) - c_\mathrm{LS}(b)  = \frac{1}{\mathrm{Area}(D) R^2} \Big(  \mathbb{E} ( N_{m^+}(D_R,[a,b])   -   N_{s^-}(D_R,[a,b]) ) \\
+ \mathbb{E} ( N_{s^+}(D_R,[a,b])- N_{m^-}(D_R,[a,b]) )  \Big). 
		\end{multline*}
	\end{theorem}
	
	It can be deduced from the above representation that $c_\mathrm{ES}$ and $c_\mathrm{LS}$ are absolutely continuous. In \cite{Beliaev2019smoothness} additional smoothness and monotonicity properties of $c_\mathrm{ES}$ and $c_\mathrm{LS}$ were derived; for instance, it was shown that $c_\mathrm{ES}(\ell)$ and $c_\mathrm{LS}(\ell)$ are continuously differentiable in $\ell$ for a wide class of fields. 
	
	\subsection{Fluctuations of the number of level/excursion set components}
	Our main results concern the order of fluctuations of $N_\mathrm{ES}$ and $N_\mathrm{LS}$. To formalise this concept we make use of the following definition, taken from \cite{chatterjee2017general}.
	
	\begin{definition}
		\label{d:fluctuations}
		Let $X_n$ be a sequence of random variables and $u_n$ a sequence of positive real numbers. We say that $X_n$ has \textit{fluctuations of order at least} $u_n$ if there exist $c_1,c_2>0$ such that, for all sufficiently large $n$ and all real numbers $a\leq b$ with $b-a\leq c_1u_n$, 
		\[ \mathbb{P}(a\leq X_n\leq b)\leq 1-c_2 . \]
		Similarly, we say that a collection of random variables $(X_R)_{R\geq 0}$ has \textit{fluctuations of order at least} $(u_R)_{R \geq 0}$ if, for any increasing sequence $R_n\to\infty$, $X_{R_n}$ has fluctuations of order at least~$u_{R_n}$.
	\end{definition}
	
	It is easy to see that if a collection of random variables $(X_n)_{n \ge 0}$ has fluctuations of order at least $(u_n)_{n \ge 0}$ then it has variance of order at least $u_n^2$, i.e.\ there exists $c>0$ such that 
	\begin{equation}
	\label{e:lbv}
	\mathrm{Var}(X_n) > cu_n^2 
	\end{equation}
	for all $n$ sufficiently large. On the other hand, having fluctuations of order at least $u_n$ is a strictly stronger statement than \eqref{e:lbv}, since the latter is consistent with the bulk of the probability mass concentrating on arbitrarily small scales.
	
	We now present our main results on the fluctuations of $N_\mathrm{ES}$ and $N_\mathrm{LS}$, which are divided into three statements. The first applies to general fields, and in particular to fields that either (i) have fast correlation decay and positive spectral density at the origin, or (ii) whose spectral measure has a singularity at the origin. The second concerns the special case of the Random Plane Wave. The third treats a certain class of somewhat degenerate fields.
	
	\subsubsection{General fields}
	To state our first result we introduce some additional assumptions on the field $f$. Recall that $\kappa$ is the covariance function of $f$. Since $\kappa$ is continuous, Bochner's theorem states that it is the Fourier transform of a positive measure $\mu$ which is Hermitian (i.e.\ $\mu(A)=\mu(-A)$ for any Borel set $A$), that is, for all $x\in\R^2$
	\begin{equation}
	\label{e:sm}
	\kappa( x)=\int_{\R^2}e^{2\pi i t\cdot x}\;d\mu(t).
	\end{equation}
	We refer to $\mu$ as the \textit{spectral measure} of the field. For some of our results we will assume that $\mu$ has a density; provided it exists, we denote this by $\rho(\cdot)$.
	
	\begin{assumption}\label{a:spectral mass}There exists a neighbourhood $V\subset\R^2$ of the origin such that the spectral measure $\mu$ has density $\rho$ on $V$ and $\inf_V\rho>0$.
	\end{assumption}
	
    The simplest way to guarantee the existence of $\rho$ is to assume that $\kappa\in L^1$; in this case $\rho$ is uniformly continuous. If we additionally assume that $\int\kappa=\rho(0)>0$, then there is a neighbourhood of the origin where $\rho$ is bounded away from $0$. On the other hand, Assumption~\ref{a:spectral mass} also allows for $\rho$ to have a singularity at the origin.
    
    Under Assumption~\ref{a:spectral mass} the support of $\mu$ contains an open set, and so the Gaussian vector formed from $f$, $\nabla f$ and $\nabla^2 f$ at a finite number of distinct points is non-degenerate (see \cite[Lemma A.2]{Beliaev2019smoothness}). 
	
For the case in which the spectral measure does not have a singularity at the origin, we will need to assume some extra control over the saddle points of the field. Let $ x_0 \in \R^2$ be a saddle point of a function $g\in C^2_\mathrm{loc}(\R^2)$ such that $g$ has no other critical points at the same level as $x_0$. We say that $x_0$ is \textit{four-arm in $D_R$} if it is in the closure of two components of $\{ x\in D_R:g( x)>g( x_0)\}$ and two components of $\{ x\in D_R:g( x)<g( x_0)\}$. Under some regularity assumptions on the function $g$, this definition implies that the four level lines (or `arms') which `emerge' from $x_0$ will all hit the boundary of $D_R$ (i.e.\ they will not `join together'). For $a \le b$, let $N_\mathrm{4\mhyphen arm}(D_R,[a,b])$ be the number of saddle points of $f$ which are four-arm in $D_R$ and have level in $[a,b]$.
	
\begin{assumption}\label{a:four arm decay}
For each open rectangle $D$ centred at the origin and every $a \le b$, there exists a function $\delta_R\to0$ as $R\to\infty$ and a constant $c>0$ such that, for each $R>1$ and $a\leq a_R\leq b_R\leq b$,
\begin{displaymath}
\mathbb{E}\left(N_\mathrm{4\mhyphen arm}\left(D_R,\left[a_R,b_R\right]\right)\right)\leq c\min\left\{\delta_R R^2(b_R-a_R),R\right\} .
\end{displaymath}
\end{assumption}
	
Sufficient conditions for a field to satisfy Assumption \ref{a:four arm decay} are given in \cite[Corollary~2.12]{Beliaev2019smoothness} (this result actually gives the analogous bound for the expected number of four-arm saddle points in $B(R)$ the ball of radius $R$, but since $B(c_DR)\subset D_R\subset B(C_DR)$ for some constants $c_D,C_D>0$ and all $R>0$, the two statements are equivalent). In particular this assumption is satisfied for isotropic fields whose correlations are positive and rapidly decaying, which includes the important special case of the Bargmann-Fock field, i.e., the field with covariance function $\kappa( x)=\exp(-\lvert  x\rvert^2/2)$ (see \cite{beffara2017percolation} for background).
	
	We can now state our fluctuation lower bound for general fields. Recall that the \textit{Dini derivatives} are a generalisation of the usual derivative, and coincide in the case of continuously differentiable functions (see \eqref{e:dini1} and \eqref{e:dini2} for a formal definition).
	
	\begin{theorem}\label{t:fluctuations general}
		Let $f$ be a Gaussian field satisfying Assumptions \ref{a:minimal} and \ref{a:spectral mass} and define $g(r)= \inf_{ x\in B(2r)}\rho( x)$. Let $D\subset\R^2$ be an open rectangle centred at the origin and recall that $D_R=\{x\in\R^2:x/R\in D\}$.  Suppose further that at least one of the following holds:
		\begin{enumerate}
			\item The field $f$ satisfies Assumption \ref{a:four arm decay}, or
			\item The spectral measure $\mu$ has a singularity at the origin, i.e.\ $g(r) \to \infty$ as $r\to 0$.
		\end{enumerate}
		If $c_\mathrm{ES}$ has a positive lower Dini derivative at $\ell$ (or a negative upper Dini derivative), then $(N_\mathrm{ES}(D_R,\ell))_{R\geq0}$ has fluctuations of order at least $R\sqrt{g(1/R)}$, and hence variance of order at least $R^2g(1/R)$. The same conclusion holds if we replace $N_\mathrm{ES}$ and $c_\mathrm{ES}$ with $N_\mathrm{LS}$ and $c_\mathrm{LS}$ respectively.
	\end{theorem}
	
	\begin{remark}
		The variance lower bound $R^2g(1/R)$ interpolates between $R^2$ (if the spectral density is bounded at the origin) and $o(R^4)$ (note that $g(1/R) = o(R^2)$ since $\rho$ is integrable on a neighbourhood of the origin). This is consistent with the trivial upper bound in \eqref{e:ub}.
	\end{remark}
	
	\begin{remark}
		It is shown in \cite{Beliaev2019smoothness} that $c_\mathrm{ES}$ and $c_\mathrm{LS}$ are continuously differentiable for a wide class of Gaussian fields, and in this case the conditions on Dini derivatives in Theorem~\ref{t:fluctuations general} are equivalent to the conditions $c_\mathrm{ES}^\prime(\ell)\neq 0$ and $c_\mathrm{LS}^\prime(\ell)\neq 0$. We expect that $c_\mathrm{ES}$ and $c_\mathrm{LS}$ are continuously differentiable in general, but we lack a comprehensive proof.
		
For general fields we expect that $c_\mathrm{ES}$ and $c_\mathrm{LS}$ have non-zero derivative for all but a small finite number of levels $\ell$ (see Section \ref{ss:diss}). In fact, based on simulations (see \cite{pranav2018unexpected}) we expect $c_\mathrm{ES}$ to be unimodal and $c_\mathrm{LS}$ to be either unimodal or bimodal depending on the field. We therefore hope that Theorem \ref{t:fluctuations general} can eventually be applied to all but a finite number of levels. On the other hand, Theorem \ref{t:fluctuations general} cannot be applied directly to $N_\mathrm{LS}(D_R, 0)$, since by symmetry $c_\mathrm{LS}^\prime(0)=0$ whenever the derivative is defined.
		
		In Section \ref{ss:diss} we give some motivation for why $c_\mathrm{ES}^\prime(\ell)\neq 0$ and $c_\mathrm{LS}^\prime(\ell)\neq 0$ are, in a sense, natural conditions for a lower bound on fluctuations.
	\end{remark}
	
	\begin{remark}
		\label{r:heavytail}
		The case of spectral singularity ($g(r) \to \infty$) is closely related to the case of \textit{long-range dependence}, i.e.\ the case in which $\kappa$ decays sufficiently slowly so as not to be integrable. In particular, standard Abel/Tauberian theorems \cite[Chapter 1.4]{leo99} imply that, up to some regularity assumptions, the asymptotics $\rho( x)  \sim | x|^{-\alpha}$ as $| x|\to0$ and $\kappa( x)  \sim | x|^{\alpha - 2}$ as $| x|\to \infty$ are equivalent for $\alpha \in (0,2)$. Hence, broadly speaking, our results shows that if correlations decay polynomially with exponent $\beta \in (0,2)$, then the variance of $N_\mathrm{ES}$ and $N_\mathrm{LS}$ grow at order at least $R^{4-\beta}$. This is analogous to known results on fluctuations of `local' functionals of long-range dependent Gaussian processes and fields \cite{leo99, slud94}.
	\end{remark}
	
	\begin{remark}
		\label{r:contain}
		Recall that $N_\mathrm{ES}(D_R,\ell)$ and $N_\mathrm{LS}(D_R,\ell)$ count the number of connected components of the excursion/level sets that intersect $D_R$ but which do not intersect the boundary $\partial D_R$; a natural question is whether the result still holds if we include components which intersect the boundary (either with or without multiplicity for repeated intersections). Since the trivial upper bound on the second moment of boundary components is $O(R^2)$, this is immediate in cases in which the variance bound is of order exceeding $R^2$. While in the general case it does not follow from our stated results, our proofs can easily be adapted to cover boundary components, but for brevity we omit the details.
	\end{remark}
	
	In order to extract from Theorem \ref{t:fluctuations general} a concrete statement about the fluctuations of $N_\mathrm{ES}$ and $N_\mathrm{LS}$, one needs to show that the (Dini) derivatives of $c_\mathrm{ES}$ and $c_\mathrm{LS}$ are non-zero for particular levels. In previous work \cite{Beliaev2019smoothness} we proved monotonicity results for $c_\mathrm{ES}$ and $c_\mathrm{LS}$ implying that this condition holds for certain ranges of levels. We illustrate this with the Bargmann-Fock field.
	
	\begin{corollary}\label{c:BF}
		Let $f$ be the Bargmann-Fock field and $D\subset\R^2$ be an open rectangle centred at the origin. There exists $\epsilon>0$ (independent of $D$) such that the following holds. If $\ell\in(-\epsilon,0.64)\cup(1.03,\infty)$ then $(N_\mathrm{ES}(D_R,\ell))_{R\geq0}$ has fluctuations of order at least $R$ and hence variance of order at least $R^2$. If $\lvert\ell\rvert> 1.03$ then $(N_\mathrm{LS}(D_R,\ell))_{R\geq 0}$ has fluctuations of order at least $R$ and hence variance of order at least $R^2$.
	\end{corollary}
	\begin{proof}
		Assumptions \ref{a:minimal} and \ref{a:spectral mass} are trivially satisfied for the Bargmann-Fock field, and \cite[Corollary 2.12]{Beliaev2019smoothness} states that Assumption \ref{a:four arm decay} is also satisfied. The corollary then follows from \cite[Proposition 2.21]{Beliaev2019smoothness}, which states that $c_\mathrm{ES}^\prime(\ell)\neq 0$ and $c_\mathrm{LS}^\prime(\ell)\neq 0$ for the respective levels given above.
	\end{proof}

	For general isotropic fields satisfying some additional assumptions listed below (Assumption \ref{a:Monotonicity}), it is shown in \cite{Beliaev2019smoothness} that $c_\mathrm{ES}$ and $c_\mathrm{LS}$ are monotone for similar ranges of levels, and so we draw similar conclusions for such fields. To be a bit more precise, for each such field there exists $\epsilon>0$ such that the fluctuations of $(N_\mathrm{ES}(D_R,\ell))_{R\geq0}$ are of order at least $R$ for $\ell \in (-\epsilon,C)\cup (\sqrt{2}/\chi,\infty)$ and the fluctuations of $(N_\mathrm{LS}(D_R,\ell))_{R\geq0}$ are of order at least $R$ for $|\ell|>\sqrt{2}/\chi$. Here $\chi$, given by \eqref{e:mon3} below, is a parameter controlling the distribution of critical points \cite{cheng2018expected} and $C$ is the positive root of an explicit but rather complicated equation involving $\chi$, the normal density function and  cumulative density function. 
	
	\begin{assumption}\label{a:Monotonicity}
		The field $f$ satisfies the following:
		\begin{itemize}
			\item $f$ is isotropic (i.e.\ its law is invariant under rotations) and
			\begin{equation}\label{e:mon3}
			\chi:=\frac{-\sqrt{3}\partial^{(2,0)}\kappa(0)}{\sqrt{\partial^{(4,0)}\kappa(0)}}\geq 1.
			\end{equation}
			\item There exist $c,\nu>0$ such that, for all $|x| \ge 1$,
			\begin{displaymath}
			\max_{\lvert \alpha\rvert \le 3} \, \left\lvert \partial^\alpha \kappa(x)\right\rvert \leq c\lvert x\rvert^{-(1+\nu)}.
			\end{displaymath}
			\item The Gaussian vector $(f(0),\nabla^2 f(0))$ is non-degenerate, and for all $x\in\R^2$,
			\begin{equation}
			\label{e:mon}
			\mathbb{E}\left(f(x) \, \middle| \, f(0)=0,\nabla^2 f(0)=\begin{pmatrix}
			1 &0\\
			0 &0
			\end{pmatrix}\right)\geq 0,
			\end{equation}
			\begin{equation}
			\label{e:mon2}
			\mathbb{E}\left(f(x) \, \middle| \, f(0)=1,\nabla^2 f(0)=\begin{pmatrix}
			0 &0\\
			0 &0
			\end{pmatrix}\right)\leq 1.
			\end{equation}
			
			\item For $0<r<R$, let $\mathrm{Arm}_\ell(r,R)$ denote the `one-arm event' that there exists a component of $\{f\geq\ell\}$ which intersects both $\partial B(r)$ and $\partial B(R)$. Then there exist $c_1,c_2>0$ such that for any $1 <r<R$
			\begin{equation}
			\label{e:arm decay}
			\mathbb{P}\left(f\in\mathrm{Arm}_{0}(r,R)\right)\leq c_1(r/R)^{c_2}.
			\end{equation}
		\end{itemize}
	\end{assumption}
	We note that the one-arm decay condition in this assumption has been verified for a wide class of fields \cite{Muirhead2018sharp,rivera2019talagrand} and is believed to hold even more generally. All other parts of this assumption can be verified directly using the covariance function. More detail on how to verify conditions \eqref{e:mon} and \eqref{e:mon2} is given in \cite{Beliaev2019smoothness}. We omit this detail here, since Assumption~\ref{a:Monotonicity} will not be used in the current paper other than for the following corollary.
	\begin{corollary}\label{c:general}
		Let $f$ satisfy Assumptions~\ref{a:minimal},~\ref{a:spectral mass} and~\ref{a:Monotonicity}. Let $D\subset\R^2$ be an open rectangle centred at the origin. There exists $\epsilon>0$ and an explicit constant $C>0$ (both independent of $D$) such that the following holds. If $\ell\in(-\epsilon,C)\cup(2/\chi,\infty)$ then $(N_\mathrm{ES}(D_R,\ell))_{R\geq0}$ has fluctuations of order at least $R$ and hence variance of order at least $R^2$. If $\lvert\ell\rvert>2/\chi$ then $(N_\mathrm{LS}(D_R,\ell))_{R\geq 0}$ has fluctuations of order at least $R$ and hence variance of order at least $R^2$.
	\end{corollary}
	\begin{proof}
		Under these assumptions, \cite[Proposition~2.22]{Beliaev2019smoothness} states that $c_\mathrm{ES}$ and $c_\mathrm{LS}$ are continuously differentiable and that
		\begin{displaymath}
		c_\mathrm{ES}^\prime(\ell)\begin{cases}
		>0 &\text{for }\ell\in (-\epsilon,C)\\
		<0 &\text{for }\ell\in\left(\sqrt{2}/\chi,\infty\right)
		\end{cases}
		\end{displaymath}
		and
		\begin{displaymath}
		c^\prime_\mathrm{LS}(\ell)<0\;\text{ for }\ell\in\left(\sqrt{2}/\chi,\infty\right)
		\end{displaymath}
		for $\epsilon,C>0$ as above. Since $c_\mathrm{LS}$ is symmetric in $\ell$, the result then follows from Theorem~\ref{t:fluctuations general}.
	\end{proof}
	The general expression for the constant $C$ in this result is given in the proof of \cite[Proposition~2.22]{Beliaev2019smoothness}. The expression is straightforward to evaluate numerically for any particular field (in particular, it depends only on the first two derivatives of the covariance function at the origin).
	
	\subsubsection{The Random Plane Wave}
	We now turn to the important special case of the Random Plane Wave (RPW), i.e.\ the field with covariance function $\kappa( x)=J_0(\lvert x\rvert)$, where $J_0$ is the \mbox{$0$-th} Bessel function. The RPW has applications in quantum chaos as a model for high energy eigenfunctions of the Laplacian (see \cite{berry1977regular}) and the geometry of its excursion/level sets have been studied by many authors (see~\cite{bogomolny2002percolation,jain2017nodal}). 
	
	The RPW does not fall within the scope of Theorem \ref{t:fluctuations general} since it does not satisfy Assumption~\ref{a:spectral mass} (its spectral measure is supported on the unit circle). Nevertheless we can prove the following bound on fluctuations.
	
	\begin{theorem}\label{t:fluctuations RPW}
		Let $f$ be the Random Plane Wave and let $D\subset\R^2$ be an open rectangle centred at the origin. If $c_\mathrm{ES}$ has a positive lower Dini derivative at $\ell\neq 0$ (or a negative upper Dini derivative), then $(N_\mathrm{ES}(D_R,\ell))_{R\geq 0}$ has fluctuations of order at least $R^{3/2}$, and hence variance of order at least $R^3$. The same conclusion holds if we replace $N_\mathrm{ES}$ and $c_\mathrm{ES}$ with $N_\mathrm{LS}$ and $c_\mathrm{LS}$ respectively.
	\end{theorem}
	
	\begin{remark}
		The larger fluctuations of $N_\mathrm{ES}$ and $N_\mathrm{LS}$ for the RPW (order $R^{3/2}$ compared to the generic $R$) can be understood as a reflection of degeneracies in the RPW, which manifest in at least four ways. First, the spectral measure $\mu$ is supported on a dimension one subspace (the unit circle). Second, and directly related to the first, is that realisations of the RPW are solutions of the Helmholtz equation $\Delta f = - f$. Third, the RPW has long-range dependence, with correlations decaying only at rate $1/\sqrt{| x|}$. Fourth, when expanded in a particular orthogonal series (see \eqref{e:exp}), only order $R$ terms are required to specify the RPW in a ball of radius $R$, up to exponentially small error, compared to the generic order $R^2$ terms for a planar field. In fact, this last property is what ultimately drives our proof of Theorem~\ref{t:fluctuations RPW}.
	\end{remark}
	
	As for the Bargmann-Fock field, in previous work we verified the Dini-derivative condition for $c_\mathrm{ES}$ and $c_\mathrm{LS}$ in certain intervals (see \cite[Proposition 2.20]{Beliaev2019smoothness}). This leads to the following corollary.
	
	\begin{corollary}\label{c:Variance RPW}
		Let $f$ be the Random Plane Wave and let $D\subset\R^2$ be an open rectangle centred at the origin. If $\ell\in(-\infty,0)\cup(0,0.87)\cup[1,\infty)$ then $(N_\mathrm{ES}(D_R,\ell))_{R\geq0}$ has fluctuations of order at least $R^{3/2}$ and hence variance of order at least $R^3$. If $\lvert \ell\rvert\geq 1$ then $(N_\mathrm{LS}(D_R,\ell))_{R\geq0}$ has fluctuations of order at least $R^{3/2}$ and hence variance of order at least~$R^3$.
	\end{corollary}
	
	\subsubsection{Degenerate fields}
	Finally we consider the class of fields whose spectral measure has a delta mass at the origin. In this case, we prove that the variance attains the order of the trivial upper bound in \eqref{e:ub} for all levels.
	
	\begin{assumption}
		\label{a:spectral atom} The stationary Gaussian field $f$ has spectral measure $\nu=\alpha\delta_0+\nu^*$ where $\alpha>0$, $\delta_0$ is a delta-mass at the origin and $\nu^*$ is a (positive) measure. If $g$ is the Gaussian field with spectral measure $\nu^*$ then $g$ satisfies Assumption~\ref{a:minimal}.	
	\end{assumption}
	
	Under this assumption, the field $f$ is no longer normalised to have variance one at a point: instead $\mathrm{Var}(f(x))=1+\alpha$ and $\mathrm{Var}(g(x))=1$ for all $x\in\R^2$. This is motivated by the fact that we can represent $f$ as
	\begin{displaymath}
	f=g+\sqrt{\alpha}Z
	\end{displaymath}
	where $Z$ is a standard Gaussian variable independent of $g$. This representation follows immediately from considering the covariance function of the field on the right. In order to analyse the level sets of $f$, we apply Theorem~\ref{t:main level} to $g$ and consider the additional effect of shifting the overall level due to $Z$. It is therefore convenient to normalise $g$ as in the rest of the paper. Of course, our results apply to any stationary Gaussian field with spectral mass at the origin, one simply has to rescale the variance to match the normalisation above.
	
	\begin{proposition}\label{t:Delta mass zero}
		Let $f$ satisfy Assumption~\ref{a:spectral atom} and let $D\subset\R^2$ be an open rectangle centred at the origin. For each $\ell\in\R$, $(N_\mathrm{ES}(f;D_R,\ell))_{R\geq 0}$ has fluctuations of order at least $R^2$. Moreover, there exist positive constants $c_1(\ell),c_2(\ell)$ (independent of $D$) such that
		\begin{displaymath}
		c_1 \mathrm{Area}(D)^2R^4 <\mathrm{Var}(N_\mathrm{ES}(f;D_R,\ell))< c_2 \mathrm{Area}(D)^2 R^4
		\end{displaymath}
		for all $R>0$ sufficiently large. The same conclusions hold if we replace $N_\mathrm{ES}$ and $c_\mathrm{ES}$ with $N_\mathrm{LS}$ and $c_\mathrm{LS}$ respectively.
	\end{proposition}
	This result roughly says that adding a random independent level shift to any non-degenerate Gaussian field (which is equivalent to adding a delta mass to the spectral measure at the origin - see above) ensures that the number of excursion/level set components of the resulting field has variance of maximal order at all levels.
	
	We note that this proposition makes no requirement on the derivative of the mean functional at a given level. Intuitively this holds because the variable $Z$ can always shift the field $g$ to levels at which the asymptotic density of excursion (or level) components differs. Although this result follows from the methods we utilise throughout the rest of the paper (to be described in Section~\ref{ss:Outline}) we actually prove it using more elementary techniques. This result is therefore included primarily for completeness and comparison, rather than as a significant contribution.
	
	\begin{remark}
		There are other degenerate classes of fields for which the variance of $N_\mathrm{ES}$ and $N_\mathrm{LS}$ can be shown to be of maximal order $R^4$:
		\begin{enumerate}
			\item For fields which are doubly-periodic (i.e.\ have spectral measure which is supported on a lattice), it is evident that $N_\mathrm{ES}$ and $N_\mathrm{LS}$ have variance of order $R^4$ whenever the variance is finite and non-degenerate (known under minimal assumptions; see Remark~\ref{r:ub}).
			\item For fields with spectral measure supported on exactly four or five points, the entire distributions of $N_\mathrm{ES}$ and $N_\mathrm{LS}$ can be explicitly calculated (see \cite[Proposition~1.20]{Beliaev2018Number} or \cite[Proposition~2.1.11]{mcauley2020excursion}). In this case the variance of $N_\mathrm{ES}$ and $N_\mathrm{LS}$ can also be shown to have order $R^4$ whenever they are non-degenerate.
		\end{enumerate}
	\end{remark}

	\subsection{Further discussion and open questions}
	\label{ss:diss}
	
	In this section we discuss conjectures, open questions, and links to other models.
	
	\subsubsection{Anomalous levels}
	
	As mentioned above, we believe that the variance bounds in Theorems \ref{t:fluctuations general} and \ref{t:fluctuations RPW} are of the correct order for generic levels, with the possible exception of a finite set of `anomalous' levels, different for $c_\mathrm{ES}$ and $c_\mathrm{LS}$, at which the variance is of lower order.
	
	\begin{conjecture}
		\label{c:1}
		Suppose that $f$ satisfies Assumptions~\ref{a:minimal},~\ref{a:spectral mass} and~\ref{a:four arm decay} (e.g., the Bargmann-Fock field) and $D\subset\R^2$ is an open rectangle centred at the origin. Then for all $\ell \in \mathbb{R}$ there exists $c_\mathrm{var}(\ell)>0$ such that 
		\[   \mathrm{Var}(N_\mathrm{ES}(D_R,\ell)) \sim c_\mathrm{var}(\ell)\mathrm{Area}(D) R^2 , \]
		and the same conclusion is true for $N_\mathrm{LS}(D_R, \ell)$.
	\end{conjecture}
	
	\begin{conjecture}
		\label{c:2}
		Suppose that $f$ satisfies Assumptions \ref{a:minimal} and \ref{a:spectral mass}, and assume that there exist $\alpha \in (0, 2)$ and $r_0 > 0$ such that 
		\[\rho(x)= \lvert x\rvert^{-\alpha}
		\] 
		for all $\lvert x\rvert < r_0$. Let $D\subset\R^2$ be an open rectangle centred at the origin. Then there exists a (possibly empty) finite set $\mathcal{L} \subset \mathbb{R}$ and $c_\mathrm{var}(\ell)>0$ such that, for all $\ell \notin \mathcal{L}$,
		\[   \mathrm{Var}(N_\mathrm{ES}(D_R,\ell)) \sim  c_\mathrm{var}(\ell)\mathrm{Area}(D)^\frac{2+\alpha}{2} R^{2+\alpha} \]
		whereas for all $\ell \in \mathcal{L}$,
		\[   \mathrm{Var}(N_\mathrm{ES}(D_R,\ell)) \ll R^{2+\alpha} ,\]
		and the same conclusion is true for $N_\mathrm{LS}(D_R, \ell)$ (with a different set $\mathcal{L}$). If $f$ is the RPW, then the same conclusion is true with $2+\alpha$ replaced with~$3$.
	\end{conjecture}
	
	These conjectures are motivated by a comparison with the known behaviour of the variance of the \textit{Minkowski functionals} of the excursion sets, namely the volume of the excursion set, the length of the level set, and the Euler characteristic of the excursion set (by Hadwiger's theorem, these form a linear basis for the set of `local' functionals of the excursion sets that are isometrically invariant \cite{RFG}). To illustrate this, let $L(R, \ell)$ denote the length of the level set $\{f = \ell\}$ contained within $[-R,R]^2$. It is known \cite{kratz18}, that for fields with rapid correlation decay, there exists $c_\mathrm{var}(\ell)>0$ such that
	\[ \mathrm{Var}(L(R, \ell)) \sim c_\mathrm{var}(\ell) R^2  , \]
	whereas for the RPW
	\begin{equation}
	\mathrm{Var}(L(R, \ell)) \sim  \begin{cases} 
	c_\mathrm{var}(\ell) R^3 & \text{for } \ell \neq 0 , \\
	c_\mathrm{var}(0) R^2 \log R & \text{for } \ell = 0 .
	\end{cases}
	\end{equation}
	In other words, for the RPW, $L(R, \ell)$ has variance of lower order at level $\ell = 0$ compared to $\ell \neq 0$. This phenomenon was first predicted by Berry \cite{berry2002statistics}, and has since been proven rigorously \cite{wigman2010fluctuations, npr}. A similar phenomenon is also known to occur for the volume of the excursion sets and the Euler characteristic (see \cite{mw11b, cammarota2016}); in the latter case the variance reduction also occurs at certain non-zero levels.
	
	The phenomenon of variance reduction can be understood as reflecting the fact that, for the RPW, the fluctuations of the Minkowski functionals are dominated by the second term in their Wiener chaos expansion, whose coefficient as a function of $\ell$ happens to vanish at certain levels (see the discussion in \cite{npr}). The same is also known to be true in the case of spectral singularity at the origin \cite[Chapter~3]{leo99}. In contrast, for fields with rapid correlation decay, many terms in the Wiener chaos expansion have fluctuations of leading order (see for instance \cite{esle16}), and so one should not expect anomalous levels since that would require many coefficients to vanish simultaneously.
	
	\subsubsection{Further questions}
	Assuming that Conjectures \ref{c:1} and \ref{c:2} are correct, they give rise to a number of further questions. For simplicity we discuss only the case of the excursion sets, but the analogous questions can be asked of the level sets. 
	
	A first set of questions concerns the anomalous levels $\mathcal{L}$ in the case of the RPW or fields with spectral singularity.
	
	\begin{samepage}
		\begin{question}\leavevmode
			\begin{enumerate}
				\item Is the set of anomalous levels $\mathcal{L}$ non-empty? What is its cardinality?
				\item Let $\mathcal{C}$ denote the set of critical points of the density functional $c_\mathrm{ES}$. By Theorems \ref{t:fluctuations general} and \ref{t:fluctuations RPW}, we know that $\mathcal{L} \subseteq \mathcal{C} \cup \{0\}$ for the RPW, whereas $\mathcal{L} \subseteq \mathcal{C}$ in the case of spectral singularity case. Are these containments strict?  
				\item What is the order of $\mathrm{Var}(N_\mathrm{ES}(D_R,\ell))$ for $\ell \in \mathcal{L}$? Does it depend on the field and on the level? Is it always of order at least $R^2$?
			\end{enumerate}
		\end{question}
	\end{samepage}
	
	Based on simulations we expect that $c_\mathrm{ES}$ is unimodal for general fields, which would imply that $|\mathcal{C}| = 1$ and so $|\mathcal{L}| \le 1$ (or $|\mathcal{L}| \le 2$ for the RPW). On the other hand we expect that $c_\mathrm{LS}$ is either unimodal or bimodal, depending on the field, which would imply that $|\mathcal{L}| \le 3$.
	
	A second question concerns the constants $c_\mathrm{var}(\ell)$ for generic levels $\ell \notin \mathcal{L}$. For the Minkowski functionals of the RPW, it is known that $c_\mathrm{var}(\ell)$ is related to the derivative of the first moment (i.e.\ density) functional $c(\ell)$ via 
	\begin{equation}
	\label{e:der}
	c_\mathrm{var}(\ell) \propto (\ell \, c'(\ell))^2 ; 
	\end{equation}
	see the formulas and discussion presented in \cite{cammarota2016, cammarota2018} (actually \eqref{e:der} has only been proven for the related model of the random spherical harmonics, but we expect it to hold also for the RPW). In particular, levels are anomalous precisely when either $\ell = 0$ or $c'(\ell) = 0$, which are exactly the conditions for which our bound in Theorem \ref{t:fluctuations RPW} hold. This is evidence that our conditions in Theorem \ref{t:fluctuations RPW} are quite natural.
	
	We are not aware of any similar results to \eqref{e:der} for the Minkowski functionals of general fields, and indeed in general it is difficult to compute the value of $c_\mathrm{var}$ exactly (even if the density $c(\ell)$ is well-understood for Minkowski functionals~\cite{RFG}). It would be interesting to know if~\eqref{e:der}, or a similar relationship, holds in more generality.
	
	\begin{question}
		What is the relationship between $c_\mathrm{var}(\ell)$ and the derivative of the density functional $c'_\mathrm{ES}(\ell)$? Is $c_\mathrm{var}(\ell) \propto (\ell \, c'_\mathrm{ES}(\ell))^2$ for the RPW?
	\end{question}
	
	The third question involves the asymptotic distribution of the fluctuations of $N_\mathrm{ES}(D_R, \ell)$. For the Minkowski functionals these are known to be Gaussian in many cases (see, e.g., \cite{mw11, esle16, cammarota2016, muller17, kratz18}). Non-Gaussian limit theorems have also been observed in the case of spectral singularity at the origin \cite{leo99, slud94}.
	
	\begin{question}  
		Does $N_\mathrm{ES}(D_R,\ell)$ have asymptotically Gaussian fluctuations? Does it depend on the field and on the level?
	\end{question}
	
	\subsubsection{Comparison to percolation models}
	Recent work has established that, in many cases, the geometry of Gaussian excursion/level sets exhibits the same behaviour as the `clusters' in discrete percolation models \cite{1996boundedness, beffara2017percolation,rivera2017critical,Muirhead2018sharp}; in particular, this is known for Gaussian fields whose correlations are positive and rapidly decaying, and has been conjectured by Bogomolny and Schmit to be true for the RPW \cite{bogomolny2002percolation}. It is therefore of interest to compare our results to what is known for percolation models.
	
	For Bernoulli percolation on $\mathbb{Z}^2$ with connection probability $p \in (0,1)$ (see \cite{grimmett1999percolation} for background on this and other percolation models), it is known that the variance of the number of clusters in the square of side-length $R$ is of order exactly $R^2$. This matches the order of our lower bound on $\mathrm{Var}(N_\mathrm{ES}(D_R,\ell))$ for Gaussian fields with positive spectral measure and rapid correlation decay, but is inconsistent with our bounds in the case of the RPW or fields with spectral singularity. In particular, our results are inconsistent with some of the stronger claims of the Bogolmony-Schmit conjecture \cite{bogomolny2002percolation,bogomolny2007random}, which imply that the variance of $N_\mathrm{ES}(D_R,\ell)$ and $N_\mathrm{LS}(D_R, \ell)$ are of order $R^2$ for the RPW when $\ell$ is close to zero. On the other hand, the most important case of the Bogolmony-Schmit conjecture is the critical case, which posits that the nodal set $\{f = 0 \}$ of the RPW has statistics that match critical Bernoulli percolation ($p = 1/2$). Unfortunately our results do not cover this case.
	
	\begin{question}  
		What is the order of $\mathrm{Var}(N_\mathrm{LS}(D_R,0))$ and $\mathrm{Var}(N_\mathrm{ES}(D_R,0))$ for the RPW? Does it agree with the Bogolmony-Schmit prediction of order $R^2$?
	\end{question}
	
	
	\subsection{Outline of the method}\label{ss:Outline}
	In this section we give an outline of the proofs of our main results (Theorems \ref{t:fluctuations general} and \ref{t:fluctuations RPW}). For clarity we focus only on the bounds for $N_\mathrm{ES}$; the proof for $N_\mathrm{LS}$ is near identical. 
	
	The foundation of our arguments is a versatile, elementary lemma due to Chatterjee.
	
	\begin{lemma}[{\cite[Lemma~1.2]{chatterjee2017general}}]\label{l:chatterjee}
		Let $X$ and $Y$ be random variables defined on the same probability space. Then, for real numbers $a\leq b$,
		\begin{displaymath}
		\mathbb{P}(a\leq X\leq b)\leq \frac{1}{2}\left(1+\mathbb{P}\left(\lvert X-Y\rvert\leq b-a\right)+d_\mathrm{TV}(X,Y)\right),
		\end{displaymath}
		where $d_\mathrm{TV}$ denotes the total variation distance between the distributions of $X$ and $Y$.
	\end{lemma}
	
	We combine this with the following definition:
	
	\begin{definition}\label{d:differ}
		Let $X_n$ and $Y_n$ be sequences of random variables defined on the same probability space and let $u_n$ be a sequence of positive real numbers. We say that $X_n$ and $Y_n$ differ by order at least $u_n$ if there exist constants $c_1,c_2>0$ such that
		\begin{displaymath}
		\mathbb{P}\left(\lvert X_n-Y_n\rvert\geq c_1u_n\right)\geq c_2
		\end{displaymath}
		for all $n$ sufficiently large.
	\end{definition}

	\begin{corollary}
		\label{c:chat}
		Let $X_n$ and $Y_n$ be sequences of random variables defined on the same probability space and let $u_n$ be a sequence of positive numbers. If $X_n$ and $Y_n$ differ by order at least $u_n$ and $d_\mathrm{TV}(X_n,Y_n)\to 0$ as $n\to\infty$, then $X_n$ has fluctuations of order at least $u_n$.
	\end{corollary}
	
	We will apply Corollary~\ref{c:chat} with $X_R = N_\mathrm{ES}(D_R, \ell)$ and $Y_R = N_\mathrm{ES}(D_R, \ell+a_R)$ for a certain sequence $a_R\to 0$ as $R\to\infty$. There are two competing requirements on $a_R$: (i) $a_R$ should decay slowly enough that $N_\mathrm{ES}(D_R, \ell)$ and $N_\mathrm{ES}(D_R, \ell+a_R)$ differ by a large order; and (ii) $a_R$ must decay quickly enough that ${d_\mathrm{TV}(N_\mathrm{ES}(D_R, \ell), N_\mathrm{ES}(D_R, \ell+a_R))}$ tends to zero. 
	Let us consider first the order by which $N_\mathrm{ES}(D_R, \ell)$ and $N_\mathrm{ES}(D_R, \ell+a_R)$ differ. Using the assumption that $c_\mathrm{ES}$ has non-zero (Dini-)derivative at $\ell$, we show in Lemma~\ref{l:fm} that
	\begin{displaymath}
	\lvert \mathbb{E}(N_\mathrm{ES}(D_R,\ell)-N_\mathrm{ES}(D_R,\ell+a_R))\rvert\gtrsim R^2a_R.
	\end{displaymath}
	Using a bound on the second moment of the number of critical points in a shrinking height window from \cite{muirhead2019second}, (proven using the Kac-Rice theorem) we then show in Lemma~\ref{l:sm} that
	\begin{displaymath}
	\left(\mathbb{E}\left[(N_\mathrm{ES}(D_R,\ell)-N_\mathrm{ES}(D_R,\ell+a_R))^2\right]\right)^{1/2}\lesssim R^2 a_R.
	\end{displaymath}
	Since these bounds are of the same order, the second moment method implies that $N_\mathrm{ES}(D_R,\ell)$ and $N_\mathrm{ES}(D_R,\ell+a_R)$ differ by order at least $R^2 a_R$ (see Proposition~\ref{p:smm}).
	
	The next step is to bound the total variation distance between $N_\mathrm{ES}(D_R, \ell+a_R)$ and $N_\mathrm{ES}(D_R, \ell)$. Our arguments in this step are different for general fields and for the RPW.
	
	For general fields (i.e.\ those satisfying the conditions of Theorem~\ref{t:fluctuations general}), our approach is to view the random variable $N_\mathrm{ES}(D_R,\ell+a_R)$ as the number of excursion sets of the field $f-a_R$ at level $\ell$, and so
	\begin{displaymath}
	d_\mathrm{TV}(N_\mathrm{ES}(D_R,\ell),N_\mathrm{ES}(D_R,\ell+a_R))\leq d_\mathrm{TV}(f,f-a_R).
	\end{displaymath}
	A Cameron-Martin argument then gives an upper bound on this distance in terms of the norm of an (approximately) constant function in the reproducing kernel Hilbert space induced by the field (see \eqref{e:rkhs} for the definition of this Hilbert space). By bounding this norm in terms of the behaviour of the spectral measure at the origin, we can prove that the total variation distance $d_\mathrm{TV}(N_\mathrm{ES}(D_R,\ell),N_\mathrm{ES}(D_R,\ell+a_R)) \to 0$ provided that
	\[ a_R  \ll  \sqrt{g(1/R)}/R .\]
	Combining this with the previous step, we deduce a fluctuation bound of order 
	\[ R^2 a_R \approx  R \sqrt{g(1/R)}. \]
	
	In the case of the RPW (Theorem \ref{t:fluctuations RPW}), the previous approach fails since non-zero constant functions cannot be approximated in the reproducing kernel Hilbert space of the RPW (which consists of solutions to the Helmholtz equation $\Delta f = -f$). Instead, our approach is to view $N_\mathrm{ES}(D_R,\ell+a_R)$ as the number of excursion sets of the field $\ell/(\ell+a_R)f$ at level $\ell$ (note that this only holds for $\ell\neq 0$, which is why the nodal level is excluded from our results on the RPW), and so
	\begin{equation}
	\label{e:tvrpw}
	d_\mathrm{TV}(N_\mathrm{ES}(D_R,\ell),N_\mathrm{ES}(D_R,\ell+a_R))\leq d_\mathrm{TV} \left(f,\frac{\ell}{\ell+a_R}f \right).
	\end{equation}
	Using an orthogonal expansion for the RPW in terms of Bessel functions \eqref{e:exp}, we show that the topological behaviour of the RPW on $D_R$ is essentially determined by $4R$ i.i.d.\ standard Gaussian variables. Pinsker's inequality therefore allows us to bound \eqref{e:tvrpw} in terms of the Kullback-Leibler divergence from one Gaussian vector to another.  We recall that, for two probability measures $P,Q$ such that $P$ is absolutely continuous with respect to $Q$, the Kullback-Leibler divergence from $Q$ to $P$ is defined as
	\begin{equation}\label{e:KL definition}
	d_\mathrm{KL}(P\;||\;Q)=\int_\Omega\log\left(\frac{dP}{dQ}\right)\;dP
	\end{equation}
	where $\Omega$ is the sample space of $P$ and $\frac{dP}{dQ}$ is the Radon-Nikodym derivative of $P$ with respect to $Q$. This quantity can be computed explicitly for Gaussian vectors, and as a result we show that $d_\mathrm{TV}(N_\mathrm{ES}(D_R,\ell),N_\mathrm{ES}(D_R,\ell+a_R)) \to 0$ provided
	\[ a_R  \ll  1/\sqrt{R} .\]
	Combining this with the previous step, we deduce a fluctuation bound of order 
	\[ R^2 a_R \approx  R^{3/2} . \]
	
	In both cases the main technical step is to ensure that the approximations (in the general case, approximating the constant function inside the reproducing kernel Hilbert space, and for the RPW, truncating the orthogonal expansion) do not radically change the number of excursion set components. To achieve this we apply Morse theory arguments to bound the change by the number of `quasi-critical points', which we can control with local computations (see the proofs of Lemmas \ref{l:Spectral density perturbation} and \ref{l:RPW perturbation} in Section \ref{s:Perturbation}).
	
	Note that the requirement that $c_\mathrm{ES}$ has non-zero (Dini)-derivative is seemingly crucial to this method. In particular, it is not possible to obtain a (weaker) lower bound on the fluctuations in the case that $c'_\mathrm{ES}(\ell) = 0$, even if we assume $c''_\mathrm{ES}(\ell) \neq 0$, since then the second moment method fails completely (the orders of the first and second moment bounds do not match).
	
	On the other hand, there are at least three ways in which one might try to extend our results using the described method:
	\begin{enumerate}
		\item First, one could prove that $c_\mathrm{ES}$ has a non-zero derivative for a larger range of levels than those in Corollaries~\ref{c:BF} and~\ref{c:Variance RPW}.
		\item Second, one could find other ways of bounding the total variation distance between the number of excursion sets at different levels (although we expect that our bounds are of the correct order).
		\item Third, one could find different variables to compare in Chatterjee's lemma. Our choice of $Y_R = N_\mathrm{ES}(D_R,\ell + a_R)$ was motivated by previous results which made an analysis of $X_R - Y_R$ tractable, but perhaps other choices of $Y_R$ might work.
	\end{enumerate}
	We also believe that this method could be useful to prove fluctuations bounds on other `non-local' (or even `local') geometric functionals of Gaussian fields, and in principle works equally well for Gaussian fields in higher dimensions or on manifolds.
	
	\section{Fluctuations of the number of excursion/level set components}\label{s:lower bound fluctuations}
	In this section we prove our main results (Theorems~\ref{t:fluctuations general} and~\ref{t:fluctuations RPW}) following the outline given in Section~\ref{ss:Outline}, subject to two auxiliary results (Lemmas~\ref{l:Spectral density perturbation} and~\ref{l:RPW perturbation}) whose proof is deferred to Section~\ref{s:Perturbation}. We also give a proof of Proposition~\ref{t:Delta mass zero} (in Section~\ref{ss:spectral atom}, which does not rely on the other results in Sections~\ref{s:lower bound fluctuations} and~\ref{s:Perturbation}).
	
	Recall that the lower and upper right Dini-derivatives of a function $g:\R\to\R$ at a point~$x$ are defined respectively as
	\begin{equation}
	\label{e:dini1}
	\partial_+g(x):=\liminf_{\epsilon\downarrow 0}\frac{g(x+\epsilon)-g(x)}{\epsilon}\quad\text{and}\quad \partial^+g(x):=\limsup_{\epsilon\downarrow 0}\frac{g(x+\epsilon)-g(x)}{\epsilon}.
	\end{equation}
	The lower and upper left Dini-derivatives are defined respectively as
	\begin{equation}
	\label{e:dini2}
	\partial_-g(x):=\liminf_{\epsilon\downarrow 0}\frac{g(x)-g(x-\epsilon)}{\epsilon}\quad\text{and}\quad \partial^-g(x):=\limsup_{\epsilon\downarrow 0}\frac{g(x)-g(x-\epsilon)}{\epsilon}.
	\end{equation}
	For the sake of simplicity, in this section we focus on $N_\mathrm{ES}$ rather than $N_\mathrm{LS}$, and we also assume the level $\ell$ is such that either $\partial^+c_\mathrm{ES}(\ell)<0$ or $\partial_+c_\mathrm{ES}(\ell)>0$ rather than one of the corresponding conditions for left Dini-derivatives. The arguments are near identical in all of these cases, and we will mention any points of difference.
	
	\subsection{Varying the level}
	We first show that $N_\mathrm{ES}(D_{R_n},\ell)$ and $N_\mathrm{ES}(D_{R_n},\ell+a_n)$ differ by at least a certain order, for carefully chosen sequences $R_n \to \infty$ and $a_n \to 0$. There are two main inputs into this result. 
	
	The first is a deterministic topological link between $N_\mathrm{ES}(D_{R_n},\ell)$ and $N_\mathrm{ES}(D_{R_n},\ell+a_n)$ derived in \cite{mcauley2020excursion}. Recall the definition of $N_{h}(D_R,[a,b])$ for $h = m^+, m^-, s^+, s^-$ given prior to Theorem~\ref{t:integral equality}. We let $N_\mathrm{crit}(D_R,[a,b])$ denote the number of critical points of $f$ in $D_R$ with level in $[a,b]$, and let $N_\mathrm{tang}(D_R)$ and $N_\mathrm{tang}(D_R,[a,b])$ denote respectively the number of critical points of $f|_{\partial D_R}$ and those with level in $[a,b]$. Recall also the definition of the number of four-arm saddles $N_\mathrm{4\mhyphen arm}(D_R,[a,b])$ given before Assumption~\ref{a:four arm decay}.
	
	\begin{lemma}[{\cite[Corollary~2.4.7]{mcauley2020excursion}}]
		\label{l:morse}
		Let $f$ be a Gaussian field satisfying Assumption~\ref{a:minimal} and let $D\subset\R^2$ be an open rectangle centred at the origin. Then there exist absolute constants $c_1, c_2>1$ such that, for all $R>0$ and $a<b$, almost surely
		\begin{align}
		\label{e:m1}
		\lvert(N_\mathrm{ES}(D_R,a)-N_\mathrm{ES}(D_R,b)) & - ( N_{m^+}(D_R,[a,b]) - N_{s^-}(D_R,[a,b])  ) \rvert \\
		\nonumber & \qquad \le c_1 \big(N_\mathrm{tang}(D_R,[a,b] )  +  N_\mathrm{4\mhyphen arm}(D_R,[a,b]) \big)  \\
		\nonumber & \qquad \le c_2 N_\mathrm{tang}(D_R)
		\end{align}
		so in particular
		\begin{equation}
		\label{e:m4}
		| N_\mathrm{ES}(D_R,a)-N_\mathrm{ES}(D_R,b) |  \le c_1\left( N_\mathrm{crit}(D_R,[a,b]) + N_\mathrm{tang}(D_R,[a,b]  ) \right). 
		\end{equation}
	\end{lemma}
	
	\begin{remark}
		For the analogous statements for level sets (which is also given in \cite[Corollary~2.4.7]{mcauley2020excursion}), the quantity 
		\[ N_{m^+}(D_R,[a,b]) - N_{s^-}(D_R,[a,b]) \]
		in \eqref{e:m1} should be replaced with
		\[  N_{m^+}(D_R,[a,b]) - N_{s^-}(D_R,[a,b]) - N_{m^-}(D_R,[a,b]) + N_{s^+}(D_R,[a,b]) . \]
	\end{remark}
	
We note that this lemma was technically stated in the case that $D$ is a ball centred at the origin, however the proof in \cite{mcauley2020excursion} also holds for rectangles. We also mention that this lemma essentially follows from the proof of \cite[Lemma~2.5]{Beliaev2018Number} (which gives the weaker inequality above and is key to proving Theorem~\ref{t:integral equality}).
	
	The second input is a moment bound on the number of critical/tangent points of $f$ in $D_R$ inside shrinking height windows, which was proven in \cite{muirhead2019second}.
	
	\begin{proposition}
		\label{p:vb}
		Let $f$ be a Gaussian field satisfying Assumption~\ref{a:minimal} and $D\subset\R^2$ a rectangle centred at the origin. Then there exists $c>0$ such that for all $R>0$ and $a<b$
		\[ \mathbb{E}\left(N_\mathrm{crit}(D_R,[a,b])^2\right) < c  \min\left\{ R^4(b-a)^2+R^2(b-a) , R^4 \right\}  \]
		and
		\[ \mathbb{E}\left(N_\mathrm{tang}(D_R,[a,b])^2\right) < c   \min\left\{ R^2(b-a)^2 + R(b-a) , R^2 \right\}   . \]
\end{proposition}
\begin{proof}
Let $c_D$ be the diameter of $D$, then $N_\mathrm{crit}(D_R,[a,b])\leq N_\mathrm{crit}(B(c_DR),[a,b])$ and \cite[Theorem 1.3]{muirhead2019second} states that the second moment of the latter quantity satisfies the first inequality above.

For the second inequality, we consider the restriction of $f$ to each of the four line segments which make up the boundary of $D_R$. The tangent points of $f$ are then the critical points of the restricted field. By Cauchy-Schwarz, it is enough to prove the inequality above for each restriction separately. This is precisely the conclusion of \cite[Theorem~A.1]{muirhead2019second}.
\end{proof}
	
	To apply the second moment method, we require a lower bound on the mean of the difference $N_\mathrm{ES}(D_{R_n},\ell) - N_\mathrm{ES}(D_{R_n},\ell+a_n)$.
	
	\begin{lemma}
		\label{l:fm}
		Let $f$ be a Gaussian field satisfying Assumption~\ref{a:minimal} and $D\subset\R^2$ a rectangle centred at the origin. If $\partial^+c_\mathrm{ES}(\ell)<0$ or $\partial_+c_\mathrm{ES}(\ell)>0$ then there exists $c>0$ such that, for any positive sequences $a_n\to 0$ and $R_n\to\infty$,
		\begin{displaymath}
		\lvert\mathbb{E}\left(N_\mathrm{ES}(D_{R_n},\ell)-N_\mathrm{ES}(D_{R_n},\ell+a_n)\right)\rvert > c R_n^2 a_n+O(R_n)
		\end{displaymath}
		for all $n$ sufficiently large. If, in addition, $f$ satisfies Assumption \ref{a:four arm decay} then
		\begin{displaymath}
		\lvert\mathbb{E}\left(N_\mathrm{ES}(D_{R_n},\ell)-N_\mathrm{ES}(D_{R_n},\ell+a_n)\right)\rvert > c R_n^2 a_n + O(R_n a_n) + O\left(\sqrt{R_n a_n}\right)
		\end{displaymath}
		for all $n$ sufficiently large. If  right Dini derivatives are replaced with left Dini derivatives, then the same conclusion holds on replacing $N_\mathrm{ES}(D_R,\ell+a_n)$ with $N_\mathrm{ES}(D_R,\ell-a_n)$. These statements also hold if excursion sets are replaced by level sets. 
	\end{lemma}
	
	\begin{proof}
		The second inequality in \eqref{e:m1} states that
		\begin{align*}
		N_\mathrm{ES}(D_{R_n},\ell)-N_\mathrm{ES}(D_{R_n},\ell + a_n)  =& N_{m^+}(D_{R_n},[\ell,\ell+a_n))-N_{s^-}(D_{R_n},[\ell,\ell + a_n))\\
		& +O\left( N_\mathrm{tang}(D_{R_n}) \right) 
		\end{align*}
		almost surely, and so by Theorem \ref{t:integral equality},
		\begin{multline*}
		\lvert\mathbb{E}(N_\mathrm{ES}(D_{R_n},\ell)-N_\mathrm{ES}(D_{R_n},\ell+a_n))\rvert \geq
		\lvert c_\mathrm{ES}(\ell)-c_\mathrm{ES}(\ell+a_n)\rvert \cdot\mathrm{Area}(D)\cdot R_n^2\\
		+O\left(\mathbb{E}\left(N_\mathrm{tang}(D_{R_n})\right) \right) . 
		\end{multline*}
		By Proposition~\ref{p:vb} and Jensen's inequality, $\mathbb{E}\left(N_\mathrm{tang}(D_{R_n})\right)\leq cR_n$ for some $c>0$ depending only on $D$ and the distribution of $f$. Applying our assumption on the Dini-derivative of $c_\mathrm{ES}$ then proves the first part of the lemma.

The tighter inequality in \eqref{e:m1} states that
\begin{multline*}
\begin{aligned}
N_\mathrm{ES}(D_{R_n},\ell)-N_\mathrm{ES}(D_{R_n}, \ell+a_n)  =& N_{m^+}(D_{R_n},[\ell,\ell+a_n))-N_{s^-}(D_{R_n},[\ell,\ell + a_n))\\
&+ O \big(N_\mathrm{tang}(D_{R_n},[\ell,\ell +a_n]) 
\end{aligned}\\
+ N_\mathrm{4\mhyphen arm}(D_{R_n},[\ell,\ell+a_n]) \big) 
\end{multline*}
almost surely, and so by Theorem \ref{t:integral equality} and Assumption \ref{a:four arm decay}, 
\begin{equation*}
\begin{aligned}
\lvert\mathbb{E}(N_\mathrm{ES}(D_{R_n},\ell)-N_\mathrm{ES}(D_{R_n},\ell+a_n))\rvert \geq& \lvert c_\mathrm{ES}(\ell)-c_\mathrm{ES}(\ell+a_n)\rvert\cdot\mathrm{Area}(D)\cdot R_n^2\\ &+ o\left(R_n^2 a_n\right) + O\left(\mathbb{E}(N_\mathrm{tang}(D_{R_n},[\ell,\ell+a_n])) \right).
\end{aligned}
\end{equation*}
By Proposition \ref{p:vb} and Jensen's inequality,
\[  \mathbb{E}(N_\mathrm{tang}(D_{R_n},[\ell,\ell + a_n])) = O(R_n  a_n) + O\left(\sqrt{R_n a_n}\right) , \]
thus proving the lemma.
\end{proof}
	
We next prove a matching second moment bound.

\begin{lemma}
\label{l:sm}
Let $f$ be a Gaussian field satisfying Assumption~\ref{a:minimal} and $D\subset\R^2$ be an open rectangle centred at the origin. There exists $c_1> 0$ such that, for any positive sequences $a_n \to 0$ and $R_n \to \infty$,
\begin{equation}
\label{e:df1}
\mathbb{E}\left( (N_\mathrm{ES}(D_{R_n},\ell) - N_\mathrm{ES}(D_{R_n},\ell + a_n))^2 \right) < c_1 \big( R_n^4 a_n^2 + R_n^2 a_n  \big), 
\end{equation}
and the same conclusion holds for level sets.
\end{lemma}
\begin{proof}
Combine \eqref{e:m4} in Lemma \ref{l:morse} with the bounds
\[ \mathbb{E}\left(N_\mathrm{crit}(D_R,[a,b])^2\right) < c_1  \big(  R^4(b-a)^2+R^2(b-a)  \big)  \]
and
\[ \mathbb{E}\left(N_\mathrm{tang}(D_R,[a,b])^2\right) < c_1  \big( R^2(b-a)^2 + R(b-a)  \big)  \]
in Proposition \ref{p:vb}.
\end{proof}	
\begin{remark}
\label{r:ub}
By taking instead the bounds
\[ \mathbb{E}\left(N_\mathrm{crit}(D_R,[a,b])^2\right) < c   R^4   \quad \text{and} \quad \mathbb{E}\left(N_\mathrm{tang}(D_R,[a,b])^2\right) < c   R^2     \]
in Proposition \ref{p:vb}, and setting $b \equiv \infty$ in \eqref{e:m4}, the same argument also establishes the trivial upper bound \eqref{e:ub} under Assumption \ref{a:minimal}. (To be precise, we take $b\to\infty$ and apply the monotone convergence theorem to the squared number of critical/tangent points. We next apply the dominated convergence theorem to $(N_{ES}(D_R,a)-N_{ES}(D_R,b))^2$ as $b\to\infty$, using the squared number of critical and tangent points with height in $[a,\infty)$ as the dominating variable.) In fact, by applying this argument to a compact domain $B(1)$, covering $B(R)$ with $\approx R^2$ copies of $B(1)$ and controlling boundary components, we can actually derive \eqref{e:ub} without the condition that $\max_{\lvert\alpha\rvert\leq 2}\lvert\partial_\alpha\kappa( x)\rvert\to 0$ as $\lvert x\rvert\to\infty$.
\end{remark}
	
	Armed with matching first and second moment bounds, an application of the second moment method yields a lower bound on the order by which $N_\mathrm{ES}(D_{R_n},\ell)$ and $N_\mathrm{ES}(D_{R_n},\ell+a_n)$ differ.
	
	\begin{proposition}
		\label{p:smm}
		Let $f$ be a Gaussian field satisfying Assumption~\ref{a:minimal} and $D\subset\R^2$ be an open rectangle centred at the origin. Assume that $\partial^+c_\mathrm{ES}(\ell)<0$ or $\partial_+c_\mathrm{ES}(\ell)>0$ and let $a_n\to 0$ and $R_n\to\infty$ be positive sequences. If either of the following conditions hold:
		\begin{enumerate}
			\item $R_na_n\to\infty$ as $n\to\infty$;
			\item $f$ satisfies Assumption \ref{a:four arm decay}, and $R_n^2a_n$ is bounded away from zero as $n\to\infty$;
		\end{enumerate}
		then $N_\mathrm{ES}(D_{R_n},\ell)$ and $N_\mathrm{ES}(D_{R_n},\ell+a_n)$ differ by order at least $R_n^2a_n$. This statement also holds if excursion sets are replaced by level sets.
	\end{proposition}
	
	\begin{proof}
		Let 
		\begin{displaymath}
		X_n=N_\mathrm{ES}(D_{R_n},\ell)\quad\quad\text{and}\quad\quad Y_n=N_\mathrm{ES}(D_{R_n},\ell+a_n).
		\end{displaymath}
		Under either condition (1) or (2) in the statement of the proposition, Lemma \ref{l:fm} shows that there is a constant $c_1>0$ such that
		\begin{displaymath}
		\mathbb{E}|X_n-Y_n| \ge \lvert\mathbb{E}(X_n-Y_n)\rvert > c_1 R_n^2 a_n
		\end{displaymath}
		for all $n$ sufficiently large. Combining this with the Paley-Zygmund inequality
		\begin{equation}\label{e:fluctuations 1}
		\begin{aligned}
		\mathbb{P}\left(\lvert X_n-Y_n\rvert >\frac{c_1}{2}R_n^2a_n\right)& \geq\mathbb{P}\left(\lvert X_n-Y_n\rvert> \frac{1}{2}\mathbb{E}(\lvert X_n-Y_n\rvert) \right) 
		\\
		&\geq  \frac{\left(\mathbb{E}\lvert X_n-Y_n\rvert\right)^2}{4 \mathbb{E}\left(( X_n-Y_n)^2\right)}.
		\end{aligned}
		\end{equation}
		Combining this with Lemma \ref{l:sm} gives
		\begin{equation}\label{e:fluctuations 2}
		\begin{aligned}
		\frac{\left(\mathbb{E}\lvert X_n-Y_n\rvert\right)^2}{\mathbb{E}\left(( X_n-Y_n)^2\right)}\geq c_2\frac{\lvert\mathbb{E}(X_n-Y_n)\rvert^2}{R_n^4a_n^2+R_n^2a_n}\geq c_3\frac{R_n^4a_n^2}{R_n^4a_n^2+R_n^2a_n}\geq c_4>0
		\end{aligned}
		\end{equation}
		for constants $c_2,c_3,c_4>0$ and all $n$ sufficiently large. Combining \eqref{e:fluctuations 1} and \eqref{e:fluctuations 2} completes the proof.
	\end{proof}
	
	\subsection{Bounding the total variation distance and completion of the proof}
	We next bound the total variation distance between the number of excursion sets at different levels, using different arguments for the case of general fields (i.e.\ fields satisfying the conditions of Theorem~\ref{t:fluctuations general}) and for the RPW. This completes the proof of the main results (subject to two auxiliary lemmas, the proofs of which are deferred until Section \ref{s:Perturbation}).
	
	\subsubsection{General fields}
	We begin by recalling some general theory of Gaussian fields (for which we refer to \cite{jan}). Recall that to a continuous Gaussian field $f$ defined on $\R^d$ we can associate a Hilbert space of functions $H \subset C(\mathbb{R}^d)$ known as the \textit{reproducing kernel Hilbert space} (RKHS), or Cameron-Martin space, defined as the completion of the space of finite linear combinations of the covariance function~$\kappa$
	\begin{equation}
	\label{e:rkhs}
	\sum_{1 \le i \le n} a_i \kappa( s_i , \cdot)   \ , \quad a_i \in \mathbb{R}, s_i \in \mathbb{R}^d ,
	\end{equation}
	equipped with the inner product  
	\[   \left\langle \sum_{1 \le i \le n} a_i K(s_i, \cdot), \sum_{1 \le j \le m} a'_j K(s'_j, \cdot) \, \right\rangle_H = \sum_{\substack{1 \le i\leq n\\1\leq j \le m}}   a_i a_j'  K(s_i, s_j') . \]
	
	The importance of the RKHS for our purposes is the following corollary of the Cameron-Martin theorem.
	
	\begin{proposition}[{\cite[Corollary~3.10]{Muirhead2018sharp}}]\label{p:cameronmartin}
	Let $f$ be a continuous Gaussian field defined on some Euclidean space. For every $h \in H$,
		\[ d_\mathrm{TV}(f,f-h)\leq \frac{\|h\|_H}{\sqrt{\log 2}}. \]
	\end{proposition}
	
	If the Gaussian field $f$ is stationary, the norm $\|h\|_{H}$ can be written explicitly in terms of the spectral measure $\mu$. Indeed we can represent $H$ as the Fourier transform of $L^2_\mathrm{sym}(d\mu)$, the space of complex Hermitian  functions square integrable with respect to $\mu$. Specifically, each $h\in H$ is of the form $\mathcal{F} \left(\hat{h} \, d\mu\right)$ with a unique $\hat{h}\in L^2_\mathrm{sym}(d\mu)$, and 
	\[
	\langle h_1,h_2\rangle_H=\left\langle\hat{h}_1,\hat{h}_2\right\rangle_{L^2_\mathrm{sym}(d\mu)}.
	\]
	In particular,
	\[
	\|h\|_H= \left\| \hat{h} \right\|_{L^2_\mathrm{sym}(d\mu)}.   
	\]
	As an immediate consequence, in the case that $\mu$ has density $\rho$, we have that $\mathcal{F}\left(\hat{h} \, d\mu \right)= \mathcal{F}\left(\hat{h}\rho \, dx\right)$, i.e.\, $\hat{h}$ differs from the standard (inverse) Fourier transform $\mathcal{F}^{-1}(h)$ by division by $\rho$. If $\Omega:=\mathrm{supp}(\hat{h})$ has finite area, this implies the bound
	\begin{align}
	\label{e:rkhsnormbound}
	\| h \|^2_{H} =  \left\| \hat{h} \right\|^2_{L^2_\mathrm{sym}(d\mu)}  &  \le  \sup \{ |\mathcal{F}^{-1}(h)(x)|^2  / \rho(x) : x \in \Omega \} \, \mathrm{Area}(\Omega)  \\
	\nonumber &  \le \frac{ \sup \{ |\mathcal{F}^{-1}(h)(x)|^2  : x \in \Omega \} \, \mathrm{Area}(\Omega)  }{ \inf\{   \rho(x) : x \in \Omega \} }. 
	\end{align}
	
	We now restrict to the setting of Theorem~\ref{t:fluctuations general}, letting $f$ be a Gaussian field satisfying Assumptions~\ref{a:minimal} and~\ref{a:spectral mass}. We recall that $g(r):=\inf_{ x\in B(2r)}\rho( x)$. For each $r>0$, we define $h_r:\R^2\to\R$ by
	\begin{equation}\label{e:h definition}
	h_r( t)=\frac{1}{4r^2}\mathcal{F}\left[\mathds{1}_{[-r,r]^2}\right]( t)=\frac{\sin(2\pi rt_1)}{2\pi rt_1}\frac{\sin(2\pi rt_2)}{2\pi rt_2}.
	\end{equation}
Since $\mu$ has a density $\rho$ which is  uniformly positive near the origin, we see that for $r>0$ sufficiently small, $h_r$ is an element of the RKHS $H$. Then, by \eqref{e:rkhsnormbound},
	\begin{equation}
	\label{e:dmass2}
	\|h_r\|_H \leq \frac{1}{2r\sqrt{g(r)}} .
	\end{equation}
	
	We will use $h_r$, as $r \to 0$, to approximate the constant function $1$; if we choose positive sequences $R_n\to\infty$ and $r_n\to 0$ such that $r_nR_n\to 0$, then by a Taylor expansion,
	\begin{align}
	\label{e:h}
	\|1-h_{r_n}\|_{C^2(D_{R_n})}=O\left(r_n^2R_n^2\right)\qquad\text{as }n\to\infty.
	\end{align}
	In the next lemma we show that this approximation has a negligible effect on the number of excursion sets, i.e.\ the number of excursion sets of $f - a_n$ is well approximated by the number of excursion sets of $f - a_n h_{r_n}$ for an appropriate choice of $r_n$. We extend our previous notation $N_\mathrm{ES}(D_R, \ell)$ slightly, defining $N_\mathrm{ES}(g;D_R,\ell)$ for $g\in C^2_\mathrm{loc}(\R^2)$ to be the number of components of $\{g\geq\ell\}$ contained in $D_R$ (so that $N_\mathrm{ES}(D_R, \ell) = N_\mathrm{ES}(f; D_R, \ell)$).

\begin{lemma}\label{l:Spectral density perturbation}
Let $f$ be a Gaussian field satisfying the conditions of Theorem \ref{t:fluctuations general} and let $D\subset\R^2$ be an open rectangle centred at the origin. Fix $\ell \in \mathbb{R}$ and let $R_n$, $r_n$ and $a_n$ be sequences of positive numbers such that $R_n\to\infty$, $r_n\to 0$, $a_n\to 0$ and $r_nR_n\to 0$ as $n\to\infty$. Then there exist $c, n_0 > 0$ such that, for all $n > n_0$,
		\begin{align*}
		\mathbb{E}&\big[ \lvert N_\mathrm{ES}(f-a_n;R_n,\ell)-N_\mathrm{ES}(f-a_n h_{r_n};R_n,\ell ) \rvert  \big] < c  a_n r_n^2 R_n^4  .\end{align*}
		The same conclusion holds for level sets.
	\end{lemma}
	
	We defer the proof of Lemma \ref{l:Spectral density perturbation} until Section \ref{s:Perturbation}. The upshot is that the conclusion of Proposition~\ref{p:smm} also holds when we replace $f-a_n$ with $f-a_nh_{r_n}$.
	
	Our final ingredient for proving Theorem~\ref{t:fluctuations general} is the following elementary lemma:
	\begin{lemma}\label{l:fluctuations}
		Let $X_n$ be a sequence of random variables and $u_n$ a sequence of positive real numbers. If $X_n$ has fluctuations of order at least $\delta_nu_n$ for all positive sequences $\delta_n$ converging to zero arbitrarily slowly, then $X_n$ has fluctuation of order at least $u_n$.
	\end{lemma}
	\begin{proof}
		We note that by Definition~\ref{d:fluctuations}, $X_n$ has fluctuations of order $u_n$ if and only if $X_n/u_n$ has fluctuations of order $1$. Therefore we may assume $u_n=1$ for all $n$. 
		
		We fix some positive sequence $\delta_k\to 0$ as $k\to\infty$, and suppose that $X_n$ does not have fluctuations of order at least $1$. Then by Definition~\ref{d:fluctuations}, for each $k$ we can find $n_k>n_{k-1}$ and $a_{n_k}<b_{n_k}$ such that $b_{n_k}-a_{n_k}\leq\delta_k^2$ and
		\begin{equation}\label{e:Elementary}
		\mathbb{P}\big(a_{n_k}\leq X_{n_k}\leq b_{n_k}\big)>1-\delta^2_k.
		\end{equation}
		By assumption, $X_{n_k}$ has fluctuations of order $\delta_k$. So there exist absolute constants $c_1,c_2>0$ such that for all $k$ sufficiently large (so that $\delta_k<c_1$) we have $b_{n_k}-a_{n_k}\leq c_1\delta_k$ and hence
		\begin{displaymath}
		\mathbb{P}\big(a_{n_k}\leq X_{n_k}\leq b_{n_k}\big)\leq \delta_k.
		\end{displaymath}
		Provided $k$ is large enough, this contradicts \eqref{e:Elementary}, so we deduce that $X_n$ has fluctuations of order $1$, as required.
	\end{proof}
	
	\begin{proof}[Proof of Theorem \ref{t:fluctuations general}]
		Let $R_n\to\infty$ and $\delta_n\to 0$ be positive, monotone sequences such that $\delta_n^2 R_n\to\infty$. If $g(r)\to\infty$ as $r\to 0$ we also choose $\delta_n$ converging to zero sufficiently slowly that $\delta_n^2 \sqrt{g(\delta_n/R_n)} \to\infty$ as $n\to\infty$. We apply Proposition \ref{p:smm} with $a_n = \delta_n^2  \sqrt{g(\delta_n/R_n)} / R_n$ and deduce that $N_\mathrm{ES}(f;D_{R_n},\ell)$ and $N_\mathrm{ES}(f;D_{R_n},\ell+a_n)$ differ by order at least $R_n^2a_n=\delta_n^2 R_n\sqrt{g(\delta_n/R_n)}$. Applying Lemma \ref{l:Spectral density perturbation} with $r_n=\delta_n/R_n$ and Markov's inequality we have that, for every $\varepsilon > 0$ and $n$ sufficiently large,
\begin{align*}
\mathbb{P}\left( \left\lvert N_\mathrm{ES}(f-a_n;D_{R_n},\ell)-N_\mathrm{ES}(f-a_n h_{r_n};D_{R_n},\ell)\right\rvert  > \varepsilon R_n^2 a_n \right)& < \frac{c a_n r_n^2 R_n^4}{ \varepsilon R_n^2 a_n}\\
&= c \delta_n^2 / \varepsilon  \to  0  . 
\end{align*}
Hence $N_\mathrm{ES}(f;D_{R_n},\ell)$ and $N_\mathrm{ES}(f-a_n h_{r_n};D_{R_n},\ell)$ also differ by order at least $R_n^2a_n$. Moreover, by Proposition \ref{p:cameronmartin} and \eqref{e:dmass2}, there is a $c_0 > 0$ such that
		\begin{multline*}
		d_\mathrm{TV}(N_\mathrm{ES}(f;D_{R_n},\ell),N_\mathrm{ES}(f-a_n h_{r_n};D_{R_n},\ell)) \leq d_\mathrm{TV}(f,f-a_nh_{r_n}) \\
		\leq c_0\|a_nh_{r_n}\|_H  \le c_0 a_n / (2 r_n \sqrt{g(r_n)}) = (c_0/2) \delta_n \to 0
		\end{multline*}
where we implicitly restrict all fields to the domain $D_{R_n}$ when calculating the total variation distance. Therefore, by Corollary~\ref{c:chat}, we conclude that $N_\mathrm{ES}(D_{R_n},\ell)$ has fluctuations of order at least $R_n^2a_n$. Since $\delta_n$ can be chosen to converge to zero arbitrarily slowly, Lemma~\ref{l:fluctuations} implies that $N_\mathrm{ES}(D_{R_n},\ell)$ has fluctuations of order at least $R_n\sqrt{g(1/R_n)}$, completing the proof of the theorem.
	\end{proof}
	
	\subsubsection{The Random Plane Wave}
	We now move onto the proof of Theorem \ref{t:fluctuations RPW}. It is known that the RPW has the orthogonal expansion
	\begin{equation}
	\label{e:exp}
	f( x)=\sum_{m\in\mathbb{Z}} a_m J_{\lvert m\rvert}(r)e^{im\theta} ,
	\end{equation}
	where $(r,\theta)$ represents $ x$ in polar coordinates, $J_m$ is the $m$-th Bessel function and $a_m=b_m+ic_m=\overline{a_{-m}}$ with $b_0$, $(\sqrt{2}b_m)_{m\in\mathbb{N}}$ and $(\sqrt{2}c_m)_{m\in\mathbb{N}}$ independent standard (real) Gaussians and $c_0=0$. (This function is clearly Gaussian and can be shown to have the correct covariance structure using Graf's addition theorem for Bessel functions.) We will use a truncation of this expansion to approximately parameterise $f$ using a finite number of random variables, however we do so in a slightly unusual way: let $d_k$ be a sequence of independent standard Gaussian variables, we then represent $a_0=b_0=\sum_{k=1}^\infty 2^{-k/2}d_k$ and for $N\in\mathbb{N}$ define
\begin{equation}
\label{e:exptrunc}
f_N( x)=\sum_{1\leq k\leq N}2^{-k/2}d_kJ_0(r)+\sum_{1\leq\lvert m\rvert\leq N} a_m J_{\lvert m\rvert}(r)e^{im\theta}.
\end{equation}
The first summation here will clearly approximate $a_0J_0(r)$ with exponentially small error as $N\to\infty$. Known inequalities for Bessel functions \cite[Section 8.5, (9)]{watson1995treatise} state that, for all $\alpha \in (0,1)$, $m \ge 0$, and $r <\alpha m$,
\begin{equation}
\label{e:bessel}
\lvert J_m(r)\rvert\leq c_1e^{-c_2 m},
\end{equation}
which means that the terms beyond $m\approx 2R$ are exponentially small inside $B(R)$. So overall we see that $f_N$ will give a very accurate approximation to $f$ on appropriate domains. The reason for approximating $a_0$ incrementally (as opposed to just truncating \eqref{e:exp} after $N$ terms) is technical; it ensures that the remainder $f-f_N$ is non-degenerate at the origin (note that $J_n(0)=0$ for all $n\geq 1$) which will simplify some of our arguments.

 In the next lemma we show that these terms have a bounded effect on the number of excursion sets.
	
	\begin{lemma}\label{l:RPW perturbation}
		Let $f$ be the Random Plane Wave, let $D\subset\R^2$ be an open rectangle centred at the origin with diameter $\mathrm{Diam}(D)$. Fix $\ell^*\in \mathbb{R}$, and $\beta\in(0,1)$. Then there exist $c, n_0>0$ such that, for all $\ell\in[\ell^*-1,\ell^*+1]$, $N \ge n_0$ and $R\cdot\mathrm{Diam}(D)\leq \beta N$,
\begin{displaymath}
\mathbb{E}\left(\lvert N_\mathrm{ES}(f; D_R,\ell)-N_\mathrm{ES}(f_N; D_R,\ell)\rvert\right) < c .
\end{displaymath}
\end{lemma}
	
	The proof of Lemma \ref{l:RPW perturbation} is deferred to Section \ref{s:Perturbation}. Armed with this lemma we can complete the proof of Theorem \ref{t:fluctuations RPW}.
	
	\begin{proof}[Proof of Theorem \ref{t:fluctuations RPW}]
		Let $R_n \to \infty$\ be a positive sequence, and $a_n=\delta_n R_n^{-1/2}$ for some sequence $\delta_n>0$ which converges to zero slowly enough that $R_na_n\to\infty$ (we will eventually allow $\delta_n$ to converge to zero arbitrarily slowly). Applying Proposition~\ref{p:smm} shows that $N_\mathrm{ES}(f;D_{R_n},\ell)$ and $N_\mathrm{ES}(f;D_{R_n},\ell+a_n)$ differ by order at least $R_n^2a_n= \delta_nR_n^{3/2}$.
		
Now choose $m_n=\lceil 2\mathrm{Diam}(D)\cdot R_n\rceil$ (where $\lceil x\rceil$ denotes the least integer greater than or equal to $x$). Applying Lemma~\ref{l:RPW perturbation} and Markov's inequality shows that for any $\epsilon>0$
		\[ \mathbb{P}\left(\lvert N_\mathrm{ES}(f;D_{R_n},\ell)-N_\mathrm{ES}(f_{m_n};D_{R_n},\ell)\rvert> \epsilon R_n^2 a_n \right)\to 0 \]
and
\[ \mathbb{P}\left(\lvert N_\mathrm{ES}(f;D_{R_n},\ell+a_n)-N_\mathrm{ES}(f_{m_n};D_{R_n},\ell+a_n)\rvert> \epsilon R_n^2 a_n \right)\to 0 \]
as $n\to\infty$, so we conclude that $N_\mathrm{ES}(f_{m_n};D_{R_n},\ell)$ and $N_\mathrm{ES}(f_{m_n}; D_{R_n}, \ell+a_n)$ also differ by order at least $\delta_nR_n^{3/2}$.
		
For $n$ sufficiently large (so that $\ell+a_n$ is bounded away from zero), $f_{m_n}(t)\geq \ell+a_n$ if and only if $\frac{\ell}{\ell+a_n}f_{m_n}(t)\geq\ell$. Therefore
\begin{displaymath}
N_\mathrm{ES}(f_{m_n};D_{R_n},\ell+a_n)=N_\mathrm{ES}\left( \frac{\ell}{\ell+a_n}f_{m_n};D_{R_n},\ell \right).
\end{displaymath}
Since $f_{m_n}$ is parametrised by $3m_n$ independent standard Gaussian variables we see that
\begin{multline*}
d_\mathrm{TV}\left(N_\mathrm{ES}(f_{m_n};D_{R_n}, \ell),N_\mathrm{ES}(f_{m_n};D_{R_n},\ell+a_n) \right) \\
\begin{aligned}
&= d_\mathrm{TV} \left(N_\mathrm{ES}(f_{m_n};D_{R_n},\ell),N_\mathrm{ES} \left(\frac{\ell}{\ell+a_n}f_{m_n};D_{R_n},\ell \right) \right) \\
&\leq d_\mathrm{TV} \left( \mathcal{N}( 0,I_{3m_n}),\mathcal{N} \left(  0,\left(\frac{\ell}{\ell+a_n}\right)^2I_{3m_n} \right) \right).
\end{aligned}
\end{multline*}
By Pinsker's inequality, the square of the above quantity is at most
\begin{equation}\label{e:KL explicit}
\frac{1}{2}d_\mathrm{KL}\left(\mathcal{N} \left( 0,\left(\frac{\ell}{\ell+a_n}\right)^2I_{3m_n}\right)\;\middle|\middle|\;\mathcal{N}\left( 0,I_{3m_n} \right) \right)
\end{equation}
where $d_\mathrm{KL}$ denotes the Kullback-Leibler divergence defined by \eqref{e:KL definition}. If $P=\mathcal{N}(0,\Sigma_1)$ and $Q=\mathcal{N}(0,\Sigma_2)$ are two centred $k$-dimensional Gaussian measures, then it is a standard result that
		\begin{displaymath}
		d_\mathrm{KL}\left(P\;||\;Q\right)=\frac{1}{2}\left(\log\left(\frac{\det\Sigma_2}{\det\Sigma_1}\right)-k+\mathrm{Tr}\left(\Sigma_2^{-1}\Sigma_1\right)\right).
		\end{displaymath}
Therefore the quantity in \eqref{e:KL explicit} is equal to
		\begin{align*}
		\frac{3m_n}{4}\left(\left(\frac{\ell}{\ell+a_n}\right)^2-1-\ln \left(\left(\frac{\ell}{\ell+a_n}\right)^2 \right) \right)\leq \frac{c_1}{4}(6R_n+1)a_n^2
\end{align*}
for some constant $c_1>0$ depending only on $\ell$ (the inequality follows from a Taylor expansion of the logarithm). This bound converges to zero as $n\to\infty$ and so we can apply Corollary~\ref{c:chat} and conclude that $N_\mathrm{ES}(f_{m_n};D_{R_n},\ell)$ has fluctuations of order at least~$\delta_nR_n^{3/2}$. Applying Lemma~\ref{l:RPW perturbation}, the same conclusion is true for $N_\mathrm{ES}(f; D_{R_n},\ell)$. Since $\delta_n$ can be chosen to converge to zero arbitrarily slowly, Lemma~\ref{l:fluctuations} implies that $N_\mathrm{ES}(f; D_{R_n},\ell)$ has fluctuations of order at least $R_n^{3/2}$ as required.
	\end{proof}
	
	\subsection{Proof of Proposition~\ref{t:Delta mass zero}}\label{ss:spectral atom}
	We now give a proof of Proposition~\ref{t:Delta mass zero} that essentially follows from the law of large numbers for excursion set components (Theorem~\ref{t:main level}). We note that it is also possible to prove this result using the methods from Section~\ref{s:lower bound fluctuations}.
	
	\begin{proof}[Proof of Proposition \ref{t:Delta mass zero}]
		Let $f$ be a Gaussian field satisfying Assumption~\ref{a:spectral atom}, which means it can be represented as $f=g+\sqrt{\alpha}Z$ where $g$ has spectral measure $\nu^*$ and $Z$ is an independent standard Gaussian variable. We note that by Theorem~\ref{t:integral equality}, $c_\mathrm{ES}(\nu^*,\ell)$ is continuous in $\ell$ and tends to zero as $\ell\to\infty$. Furthermore, by \cite[Corollary~1.12]{Beliaev2018Number} $c_\mathrm{ES}(\nu^*,\ell)$ is positive for some $\ell\in\R$ (actually for any $\ell>0$). Combining these facts we see that there exists $\epsilon>0$ and open sets $U_1,U_2\subset\R$ such that for all $\ell_1\in U_1$ and $\ell_2\in U_2$
		\begin{displaymath}
		c_\mathrm{ES}(\nu^*,\ell_1)>4\epsilon\quad\text{and}\quad c_\mathrm{ES}(\nu^*,\ell_2)<\epsilon.
		\end{displaymath}
		We now choose an arbitrary $\ell\in\R$ and a positive, increasing sequence $R_n\to\infty$. Using the independence of $Z$ and $g$, and denoting the standard Gaussian probability density by $\phi$, we see that
		\begin{multline*}
		\mathbb{P}\left(N_\mathrm{ES}(f;D_{R_n},\ell)\geq 3\epsilon\mathrm{Area}(D)R_n^2\right)\\
		\begin{aligned}
		&=\mathbb{P}\left(N_\mathrm{ES}(g;D_{R_n},\ell-\sqrt{\alpha}Z)\geq 3\epsilon\mathrm{Area}(D)R_n^2\right)\\
		&=\int_\R\mathbb{P}\left(N_\mathrm{ES}(g;D_{R_n},\ell-\sqrt{\alpha}x)\geq 3\epsilon\mathrm{Area}(D)R_n^2\right)\phi(x)\;dx\\
		&=\int_\R\mathbb{P}\left(N_\mathrm{ES}(g;D_{R_n},u)\geq 3\epsilon\mathrm{Area}(D)R_n^2\right)\phi\left(\frac{\ell-u}{\sqrt{\alpha}}\right)\frac{1}{\sqrt{\alpha}}\;du\\
		&\geq\int_{U_1}\mathbb{P}\left(N_\mathrm{ES}(g;D_{R_n},u)\geq 3\epsilon\mathrm{Area}(D)R_n^2\right)\phi\left(\frac{\ell-u}{\sqrt{\alpha}}\right)\frac{1}{\sqrt{\alpha}}\;du.
		\end{aligned}
		\end{multline*}
		Since $g$ satisfies Assumption~\ref{a:minimal}, Theorem~\ref{t:main level} implies that for each $u\in U$, the probability in the integrand above converges to one as $n\to\infty$. Applying Fatou's lemma (and the fact that $U$ is open, so has positive Lebesgue measure) we see that the integral above is bounded below by a positive constant for all sufficiently large $n$.
		
Applying an identical argument to $U_2$ shows that $\mathbb{P}\left(N_\mathrm{ES}(f;D_{R_n},\ell)\leq2\epsilon\mathrm{Area}(D)R_n^2\right)$ is bounded below by a positive constant for all sufficiently large $n$. These two bounds show that $N_\mathrm{ES}(f;D_{R_n},\ell)$ has fluctuations of order at least $R_n^2$ (by definition). Since $\ell\in\R$ and $R_n\uparrow\infty$ were arbitrary, combining this with the trivial upper bound on the variance~\eqref{e:ub} (see also Remark~\ref{r:ub}) completes the proof of the result. Identical arguments apply to level sets.
	\end{proof}
	

	\section{Perturbation arguments}
	\label{s:Perturbation}
In this section we prove Lemmas~\ref{l:Spectral density perturbation} and~\ref{l:RPW perturbation}, thus completing the proof of all results in the paper.
	
We begin with some heuristics. Let $F$ be a realisation of a random field on a compact domain $D$. Our aim is to control the expected difference between the number of components of $\{F\geq 0\}$ and $\{F-p\geq 0 \}$, where $p$ is a small (possibly random) perturbation. If $p$ is a constant function taking the value $c>0$, then the standard methods of Morse theory show that the difference between these two quantities is at most the number of critical points of $F$ with level between $0$ and $c$ (because the excursion set $\{F\geq\ell\}$ varies continuously with $\ell$ unless passing through a critical point of $F$, in which case the number of components changes by at most one). Since the number of critical points is a local quantity, we can use the Kac-Rice formula to bound its mean.

This same reasoning can be applied to more general perturbations $p$. Assuming some regularity of $F$ and $p$ (which will be specified below) the number of components of $\{F-\alpha p\geq 0\}$ changes continuously with $\alpha$ unless passing through a value at which $0$ is a critical level of  $F-\alpha p$, and it can be shown that at such points the number of components changes by at most one. Therefore the difference in the number of excursion sets is bounded above by the number of points at which $F-\alpha p=0$ and $\nabla(F-\alpha p)=0$ for some $\alpha\in[0,1]$ (plus an analogous term which controls boundary effects, as we work with a bounded domain). Under our non-degeneracy assumptions, with probability one there is at most one such point for each $\alpha$, and for all but finitely many values of $\alpha$ there are no such points.  Since the number of these points is still a local quantity, an application of the Kac-Rice formula will yield Lemmas \ref{l:Spectral density perturbation} and \ref{l:RPW perturbation}.
	
Let us formalise the concepts just described. Let $\mathcal{D}\subset \R^2$ be an open rectangle and let $F,p$ be $C^2$ functions defined on a neighbourhood of $\overline{\mathcal{D}}$. Let $\mathcal{C}$ denote the set of corners of $\overline{\mathcal{D}}$ and $\mathcal{E}:=\partial\mathcal{D}\backslash\mathcal{C}$ denote the edges of $\overline{\mathcal{D}}$. We refer to $\mathcal{D}$, $\mathcal{E}$ and $\mathcal{C}$ as the strata of $\overline{\mathcal{D}}$. We say that $x$ in a neighbourhood of $\overline{\mathcal{D}}$ is a quasi-critical point of $(F,p)$ at level $\alpha\in[0,1]$ if $(F-\alpha p)(x)=0$ and $x$ is a critical point of $F-\alpha p$ restricted to the stratum of $\overline{\mathcal{D}}$ (or $\R^2\backslash\overline{\mathcal{D}}$) containing $x$. So for $x\in\mathcal{D}\cup(\R^2\backslash\overline{\mathcal{D}})$, this condition says that $(F-\alpha p)(x)=0$ and $\nabla(F-\alpha p)(x)=0$. For $x\in\mathcal{E}$ this means that $(F-\alpha p)(x)=0$ and $\nabla(F-\alpha p)(x)$ is normal to $\partial\mathcal{D}$ at $x$. For $x\in\mathcal{C}$ this just says that $(F-\alpha p)(x)=0$; the other condition holds vacuously. We say that a quasi-critical point $x$ at level $\alpha$ is non-degenerate if $x$ is a non-degenerate critical point of $F-\alpha p$ restricted to the stratum containing $x$ (so for example, if $x\in\mathcal{D}$, this just says that $\det\nabla^2(F-\alpha p)(x)\neq 0$).

We let $N_\mathrm{QC}(F;p,\overline{\mathcal{D}})$ denote the number of quasi-critical points of $(F,p)$ in $\overline{\mathcal{D}}$ (at any level $\alpha\in[0,1]$). If we replace $\overline{\mathcal{D}}$ with a stratum ($\mathcal{D}$, $\mathcal{E}$ or $\mathcal{C}$) in this expression, then we mean the number of quasi-critical points restricted to that stratum. For $x \in \mathcal{E}$, let $v_\partial(x)$ and $v_{\bar \partial}(x)$ denote respectively the unit vectors in the tangent and normal directions to $\partial\mathcal{D}$, and let $\nabla_\partial$ and $\nabla_{\bar \partial}$ denote the derivatives in these respective directions.

\begin{assumption}\label{a:Morse}
Let $\mathcal{D}\subset\R^2$ be an open rectangle. Let $F$ and $p$ be $C^2$ functions defined on a neighbourhood of $\overline{\mathcal{D}}$ which satisfy the following:
\begin{enumerate}
\item The quasi-critical points of $(F,p)$ in a neighbourhood of $\overline{\mathcal{D}}$ are non-degenerate, all occur at distinct levels and are not contained in $\{p=0\}$;
\item If $x\in\partial\mathcal{D}$ is a quasi-critical point of $(F,p)$ at level $\alpha$ then $\nabla(F-\alpha p)(x)\neq 0$; if, in addition, $x\in\mathcal{C}$ then $\nabla(F-\alpha p)(x)$ is not parallel to either edge of $\mathcal{D}$ which joins at $x$.
\end{enumerate}
\end{assumption}
\begin{theorem}\label{t:Morse}
Let $\mathcal{D}\subset\R^2$ be an open rectangle and let $(F,p)$ satisfy Assumption~\ref{a:Morse}, then
\begin{displaymath}
\left\lvert N_\mathrm{ES}\left(F;\overline{\mathcal{D}},0\right)-N_\mathrm{ES}\left(F-p;\overline{\mathcal{D}},0\right)\right\rvert\leq  N_\mathrm{QC}\left(F,p,\overline{\mathcal{D}}\right).
\end{displaymath}
The same conclusion holds on replacing $N_\mathrm{ES}$ with $N_\mathrm{LS}$.
\end{theorem}

This theorem is only a slight generalisation of known results from Morse theory and so its proof is given in Appendix~\ref{a:Appendix}.

We can now state our main perturbation result.
	
\begin{proposition}\label{p:Gaussian perturbation}
Let $\mathcal{D}\subset\R^2$ be an open rectangle and $\mathcal{D}_+$ be a fixed compact neighbourhood of $\overline{\mathcal{D}}$. Suppose that $F$ and $p$ are independent $C^3$-smooth planar Gaussian fields defined on $\mathcal{D}_+$ satisfying the following conditions:
\begin{enumerate}
\item For each $x, y \in \mathcal{D}_+$, $x \neq y$, the Gaussian vector $(F(x),F(y),\nabla F(x),\nabla F(y))$ is non-degenerate;
\item For each $x\in\mathcal{D}_+$, the Gaussian vector $(\nabla F(x),\nabla^2 F(x))$ is non-degenerate;
\item For each $x \in\mathcal{E}$, the vectors $(F(x), \nabla_\partial F(x), \nabla_{\partial} \nabla_{\bar{\partial}} F(x)  )$ and $(F(x), \nabla_\partial F(x), \nabla^2_{\partial} F(x)  )$
are non-degenerate (as Gaussians);
\item Either:
\begin{enumerate}
\item $p$ is deterministic and the set $\{p=0\}\cap \mathcal{D}_+$ consists of a finite union of isolated points, or
\item For each $x \in \mathcal{D}_+$, the Gaussian vector $(p(x),\nabla p(x))$ is non-degenerate.
\end{enumerate}
\end{enumerate}
Then with probability one
\begin{displaymath}
\left\lvert N_\mathrm{ES}\left(F;\mathcal{D},0\right)-N_\mathrm{ES}\left(F-p;\mathcal{D},0\right)\right\rvert
\leq N_\mathrm{QC}\left(F;p,\overline{\mathcal{D}}\right).
\end{displaymath}
The same conclusion holds on replacing $N_\mathrm{ES}$ with $N_\mathrm{LS}$.
\end{proposition}

Given Theorem~\ref{t:Morse}, the proof of this result is a straightforward application of Bulinskaya's lemma to various combinations of $F$, $p$ and their first two derivatives.
	
	\begin{proof}
		It is sufficient to verify that $F$ and $p$ almost surely satisfy both parts of Assumption~\ref{a:Morse}, since then Theorem~\ref{t:Morse} yields the result.
		
		\noindent (1). We first verify that, almost surely, the quasi-critical points of $(F,p)$ are at distinct levels. We define $\mathcal{D}_n=\{(x,y)\in\mathcal{D}_+^2:\lvert x-y\rvert\geq 1/n\}\times[0,1]$ and $g_1:\mathcal{D}_n\to\R^6$ by
		\begin{displaymath}
		g_1(x,y,\alpha)=\begin{pmatrix}
		\nabla(F-\alpha p)(x)\\
		\nabla(F-\alpha p)(y)\\
		(F-\alpha p)(x)\\
		(F-\alpha p)(y)
		\end{pmatrix}.
		\end{displaymath}
		Bulinskaya's lemma (\cite[Lemma 11.2.10]{RFG}) states that $g_1$ almost surely does not hit $ 0\in\R^6$ at any point in $\mathcal{D}_n$ provided it is almost surely $C^1$ and that the density of $g_1(x,y,\alpha)$ is bounded on a neighbourhood of $ 0$ uniformly in $\mathcal{D}_n$. We know that $g_1\in C^1(\mathcal{D}_n)$ by assumption and so turn to the second condition. Since $F$ and $p$ are independent, the density of $g_1(x,y,\alpha)$ can be given by a convolution over the densities of
		\begin{equation}
		\label{e:41}
		(\nabla F(x),\nabla F(y),F(x),F(y)) \quad\quad\text{and}\quad\quad\alpha(\nabla p(x),\nabla p(y),p(x),p(y)),
		\end{equation}
		and therefore it is sufficient to show that the density of either of these vectors is bounded. By assumption, the covariance matrix of the first vector in \eqref{e:41} is non-degenerate for every $x\neq y$, and by continuity the determinant of this covariance matrix is bounded away from zero on the compact set $\mathcal{D}_n$. Since this vector is Gaussian, this implies that its density is uniformly bounded above and allows us to apply Bulinskaya's lemma. Taking a countable union over $n$ completes the proof for quasi-critical points in $\mathcal{D}\cup(\mathcal{D}_+\backslash\overline{\mathcal{D}})$.
		
		Applying the same argument to $g_2:\mathcal{E}^2\times[0,1]\to\R^4$ and $g_3:\mathcal{E}\times\mathcal{D}_+\times[0,1]\to\R^5$ defined by
		\begin{displaymath}
		g_2:=\begin{pmatrix}
		\nabla_\partial(F-\alpha p)(x)\\
		\nabla_\partial(F-\alpha p)(y)\\
		(F-\alpha p)(x)\\
		(F-\alpha p)(y)
		\end{pmatrix}
		\quad\quad\text{and}\quad\quad
		g_3:=\begin{pmatrix}
		\nabla_\partial(F-\alpha p)(x)\\
		\nabla(F-\alpha p)(y)\\
		(F-\alpha p)(x)\\
		(F-\alpha p)(y)
		\end{pmatrix},
		\end{displaymath}
		proves that, almost surely, the quasi-critical points of $(F,p)$ in $\mathcal{D}\cup(\mathcal{D}_+\backslash\overline{\mathcal{D}})$ and $\mathcal{E}$ are all at distinct levels. Considering similar functions shows that the quasi-critical points in the corners $\mathcal{C}$ also occur at disjoint levels, but we omit this for brevity.
		
		Applying the arguments above to $g_4:\mathcal{E}\times[0,1]\to\R^3$ given by
		\begin{displaymath}
		g_4:=\begin{pmatrix}
		\nabla^2_\partial(F-\alpha p)(x)\\
		\nabla_\partial(F-\alpha p)(x)\\
		(F-\alpha p)(x)
		\end{pmatrix}
		\end{displaymath}
		shows that the quasi-critical points of $(F,p)$ in $\mathcal{E}$ are non-degenerate almost surely. A slightly different version of Bulinskaya's lemma (\cite[Proposition 6.5]{azais2009level}) states that since $g_5:\mathcal{D}_+\times[0,1]\to\R^3$ defined by
		\begin{displaymath}
		g_5:=\begin{pmatrix}
		\nabla(F-\alpha p)(x)\\
		(F-\alpha p)(x)
		\end{pmatrix}
		\end{displaymath}
		is almost surely $C^2$ and has a uniformly bounded univariate probability density, there is almost surely no $(x,\alpha)\in \mathcal{D}_+\times[0,1]$ such that $g_5(x,\alpha)=0$ and $\det\nabla^2(F-\alpha p)(x)=0$. Hence the quasi-critical points of $(F,p)$ in $\mathcal{D}\cup(\mathcal{D}_+\backslash\overline{\mathcal{D}})$ are non-degenerate almost surely. Finally, any quasi-critical points in $\mathcal{C}$ are vacuously non-degenerate.
		
		It remains to show that $(F,p)$ has no quasi-critical points in $\{p=0\}$. In the case that $p$ is deterministic, we apply Bulinskaya's lemma to $g_5$ restricted to $(\mathcal{D}_+\cap\{p=0\})\times[0,1]$ and $g_6:(\partial\mathcal{D}\cap\{p=0\})\times[0,1]\to\R^2$ defined by
		\begin{displaymath}
		g_{6}:=\begin{pmatrix}
		\nabla_\partial(F-\alpha p)(x)\\
		(F-\alpha p)(x)
		\end{pmatrix},
		\end{displaymath}
		which gives the result. (Technically in the first case we need only consider one component of $\nabla(F-\alpha p)$ since our domain is one-dimensional.) In the case that the variance of $p(x)$ is non-zero for each $x$ (and so bounded away from zero by compactness), we define ${g_{7}:\mathcal{D}_+\times[0,1]\to\R^4}$ and $g_{8}:\mathcal{E}\times[0,1]\to\R^3$ by
		\begin{displaymath}
		g_{7}:=\begin{pmatrix}
		\nabla(F-\alpha p)(x)\\
		F(x)\\
		p(x)
		\end{pmatrix}
		\quad\quad\text{and}\quad\quad
		g_{8}:=\begin{pmatrix}
		\nabla_\partial(F-\alpha p)(x)\\
		F(x)\\
		p(x)
		\end{pmatrix}.
		\end{displaymath}
		Once again, Bulinskaya's lemma (along with the convolution argument) gives the result. For the four points in $\mathcal{C}$ we note that $p$ is non-zero almost surely.
		
		\noindent (2). Applying the same arguments to $g_{5}$ restricted to $\mathcal{E}\times[0,1]$ shows that $(F,p)$ almost surely has no quasi-critical points in $\mathcal{E}$ such that $\nabla(F-\alpha p)=0$. Similarly restricting $g_{6}$ to $\mathcal{C}\times[0,1]$ (where we define $\nabla_\partial$ in the appropriate way for each point) proves the second part of this condition.
	\end{proof}

We now prove Lemmas~\ref{l:Spectral density perturbation} and~\ref{l:RPW perturbation}. We state these proofs exclusively for excursion sets; the proofs for level sets are identical. The simpler case is Lemma \ref{l:Spectral density perturbation}, since the perturbation is deterministic.
	\begin{proof}[Proof of Lemma \ref{l:Spectral density perturbation} (given Proposition~\ref{p:Gaussian perturbation})]
		For fixed $n$ sufficiently large, we define 
		\begin{displaymath}
		F= f - \ell - a_n\qquad\text{and}\qquad p= a_n (h_{r_n} - 1).
		\end{displaymath}
		Then by the definition of $h_{r_n}$ in \eqref{e:h definition}, the zero set of $p|_{D_{R_n}}$ consists of a single point at the origin (this also requires $n$ to be sufficiently large). Hence we may apply Proposition~\ref{p:Gaussian perturbation} to the functions $F$ and $p$ on the domain $\mathcal{D} = D_{R_n}$ ($f$ satisfies the conditions of the proposition by Assumptions~\ref{a:minimal} and~\ref{a:spectral mass}), and so it is sufficient to prove that there exists a $c > 0$ such that, for $n$ sufficiently large,
\begin{equation}
\label{e:qc}
\mathbb{E} \left( N_\mathrm{QC}\left(F;D_{R_n}, p\right)\right)  < c  a_n r_n^2 R_n^4  \quad \text{and}  \quad  \mathbb{E} \left( N_\mathrm{QC}\left(F;\partial D_{R_n}, p\right) \right)  < c  a_n r_n^2 R_n^3    .
\end{equation}
We begin with the first bound in \eqref{e:qc}. Define $G: D_{R_n} \times[0,1]\to\R^3$ by
\begin{equation}
\label{e:g}
G(x,\alpha)=\begin{pmatrix}
\nabla(F-\alpha p)(x)\\
(F-\alpha p)(x) 
\end{pmatrix}
\end{equation}
and let $p_{G(x,\alpha)}$ denote the density of the (non-degenerate) Gaussian vector $G(x,\alpha)$.
		
We now apply the Kac-Rice formula to $G$. Specifically we apply \cite[Theorem~6.2 and Proposition~6.5]{azais2009level} (which require $G$ to by $C^2$ almost surely) to conclude that
		\begin{multline*}
		\mathbb{E} ( N_\mathrm{QC}(F;D_{R_n}, p) )  \\
		\begin{aligned}
		&=\iint\limits_{D_{R_n}\times[0,1]} \mathbb{E}\left(\left\lvert\det \! \begin{pmatrix}
		\nabla^2(F-\alpha p)(x) &\nabla(F-\alpha p)(x)\\
		-\nabla p(x)^t & -p(x)
		\end{pmatrix}
		\right\rvert\;\middle| G(x,\alpha)=0\right) p_{G(x,\alpha)}(0) d\alpha dx\\
		& \leq  \mathrm{Area}(D) R^2_n  \! \!
		\end{aligned}		\\
		 \cdot\sup_{\substack{x \in D_{R_n}\\\alpha \in [0, 1]}} \mathbb{E}\left(\left\lvert\det \! \begin{pmatrix}
		\nabla^2(F-\alpha p)(x) &\nabla(F-\alpha p)(x)\\
		-\nabla p(x)^t & -p(x)
		\end{pmatrix}
		\right\rvert\;\middle| G(x,\alpha)=0\right) p_{G(x,\alpha)}(0) .
		\end{multline*}
		The density $ p_{G(x,\alpha)}(0)$ is bounded above by $c_1/\sqrt{\det\Sigma(x, \alpha)}$, where $\Sigma(x, \alpha)$ is the covariance matrix of $G(x, \alpha)$ and $c_1 > 0$ is an absolute constant. Using the condition $G(x,\alpha)=0$ (which implies $\nabla(F-\alpha p)(x)=0$), we have
		\begin{multline*}
		\mathbb{E}\left(\left\lvert\det \! \begin{pmatrix}
		\nabla^2(F-\alpha p)(x) &\nabla(F-\alpha p)(x)\\
		-\nabla p(x)^t & -p(x)
		\end{pmatrix}
		\right\rvert\;\middle| G(x,\alpha)=0\right) \\
		\begin{aligned}
		&=\lvert p(x)\rvert\cdot\mathbb{E}\left(\left\lvert\det\left(\nabla^2(F-\alpha p)(x)\right)\right\rvert\;\middle| G(x,\alpha)=0\right) \\
		&\le c_2 \lvert p(x)\rvert \max_{|k|= 2} \max \left\{ \mathbb{E}\left( \left\lvert\partial^k F(x)\right\rvert^2 \middle| G(x, \alpha) = 0  \right) ,  \left\lvert\partial^kp(x)\right\rvert^2  \right\}  ,
		\end{aligned}
		\end{multline*}
		where $c_2 > 0$ is an absolute constant and we have expanded the determinant and applied H\"{o}lder's inequality in the final line. By \eqref{e:h},
		\begin{displaymath}
		\|p\|_{C^2(D_{R_n})} = O( a_n (r_n R_n)^2) \to 0\qquad\text{as }n\to\infty.
		\end{displaymath}
		Hence it suffices to show that
		\begin{equation}
		\label{e:qc2}
		\frac{1}{\det(\Sigma(x, \alpha))}\quad \text{and} \quad  \max_{|k| = 2} \mathbb{E}\left( \left\lvert\partial^k F(x)\right\rvert^2 \middle| G(x, \alpha) = 0  \right)  
		\end{equation}
		are bounded above, uniformly over $n$ sufficiently large and $(x, \alpha) \in D_{R_n} \times [0,1]$.
		
		Considering the first term; since $p$ is deterministic and $f$ is stationary,
		\begin{displaymath}
		\Sigma(x,\alpha)=\mathrm{Cov}\begin{pmatrix}
		\nabla f(x)\\f(x)
		\end{pmatrix}
		=\mathrm{Cov}\begin{pmatrix}
		\nabla f(0)\\f(0)
		\end{pmatrix}
		\end{displaymath}
		which has a non-zero determinant, by Assumption~\ref{a:minimal}, that does not depend on $x$ or $\alpha$.
		
		Now turning to the second term of \eqref{e:qc2}; since $f$ is stationary, $(f(x),\nabla^2 f(x))$ is independent of $\nabla f(x)$ (this is a standard result for stationary Gaussian fields, see \cite[Chapter~5]{RFG}) and so for $\lvert k\rvert=2$
		\begin{equation}\label{e:qc5}
		\begin{aligned}
		\mathbb{E}\left( \left\lvert\partial^k F(x)\right\rvert^2 \middle| G(x, \alpha) = 0  \right)&=\mathbb{E}\left( \left\lvert\partial^k f(x)\right\rvert^2 \middle| f(x)=\ell+a_n+\alpha p(x)  \right)\\
		&=\mathbb{E}\left( \left\lvert\partial^k f(0)\right\rvert^2 \middle| f(0)=\ell+a_n+\alpha p(x)  \right)
		\end{aligned}
		\end{equation}
where the second equality also follows from stationarity of $f$ (recall that $p$ is deterministic). By Gaussian regression (\cite[Proposition~1.2]{azais2009level}), this final expression is continuous in the variable $\ell+a_n+\alpha p(x)$. Since $a_n$ and  $\|p\|_{C^2(D_{R_n})}$ converge to zero as $n\to\infty$, we see that this variable has a uniformly bounded range for all $n$ and all $(x,\alpha)\in D_{R_n}\times[0,1]$. Therefore \eqref{e:qc5} is uniformly bounded as required. This establishes the first bound in \eqref{e:qc}.
		
We prove the second bound in \eqref{e:qc} for the number of quasi-critical points in each of the four edges (i.e.\ line segments) that make up $\partial D_{R_n}$ separately. Let $F_\partial$ and $p_\partial$ be the restrictions of $F$ and $p$ respectively to one of these line segments (viewed as one-dimensional fields). We apply the Kac-Rice formula to
		\begin{equation*}
		G_\partial(x,\alpha)=\begin{pmatrix}
		(F^\prime_\partial-\alpha p^\prime_\partial)(x)\\
		(F_\partial-\alpha p_\partial)(x) 
		\end{pmatrix}
		\end{equation*}
		which gives the result once we establish that 
		\[    \mathbb{E} \left( \left\lvert  F_\partial^{\prime\prime}(x)\right\rvert \; \middle| G_\partial(x,\alpha) = 0 \right)  \, , \quad \left\lvert p_\partial^{\prime\prime}(x)\right\rvert \quad \text{and} \quad \frac{1}{\det(\Sigma_\partial(x, \alpha))} \]
		are bounded above uniformly over $n$ sufficiently large and $(x, \alpha) \in [0,c R_n]\times [0,1]$, where $\Sigma_\partial(x,\alpha)$ denotes the covariance matrix of $G_\partial(x, \alpha)$ and we have parameterised the line segment by $[0,cR_n]$ for some $c$ depending only on $D$. These facts are proven using arguments near identical to those given above.
	\end{proof}
	
	The proof of Lemma \ref{l:RPW perturbation} is slightly more complex because the perturbation is random, and we require an additional lemma to control its behaviour:
	
	\begin{lemma}
		\label{l:RPWbound}
		Let $f$ be the Random Plane Wave and $f_N$ the approximation of this expansion given by \eqref{e:exptrunc}. For each $\beta \in (0,1)$ and $k \in \mathbb{N}$ there exists $c_1,c_2 > 0$ such that, for all $N \ge 1$,
		\begin{displaymath}
		\mathbb{E} \left( \|f-f_N\|_{C^{k}(B(\beta N))}^2 \right)\leq c_1e^{-c_2N}.
		\end{displaymath}
	\end{lemma}
	\begin{proof}
		Recall from \eqref{e:bessel} that there exist $c_1,c_2>0$ such that, for all $m \ge 0$ and $r\leq \alpha m$,
		\begin{equation}\label{e:Bessel bound}
		\lvert J_{m}(r)\rvert\leq c_1 e^{-c_2 m}.
		\end{equation}
		Applying this to the orthogonal expansion \eqref{e:exp}, along with the fact that $\lvert J_0\rvert\leq 1$ we have
		\begin{align*}
		\mathbb{E} \left(  \sup_{x\in B(\beta N)}\left\lvert f(x)-f_N(x)\right\rvert^2 \right)&\leq \mathbb{E}\left(\left(\sum_{j>N}2^{-j/2}\lvert d_j\rvert+\sum_{\lvert m\rvert>N}c_1\lvert a_m\rvert e^{-c_2\lvert m\rvert}\right)^2 \right)\\
		&\leq c_3e^{-c_4N}
		\end{align*}
		for some $c_3,c_4>0$, which gives the result in the case $k = 0$. 
		
		For the general case $k \ge 1$, we differentiate in polar coordinates and use Bessel identities to replace the resulting terms by linear combinations of Bessel functions. For instance, since  
		\begin{displaymath}
		\frac{2mJ_m(r)}{r}=J_{m-1}(r)+J_{m+1}(r)\quad\text{and}\quad 2J_m^\prime(r)=J_{m-1}(r)-J_{m+1}(r)
		\end{displaymath}
		we have that
		\begin{multline*}
		\begin{aligned}
		\partial_{x_1}f(x)-\partial_{x_1}f_N(x)=&\sum_{j>N}-2^{-j/2}d_jJ_1(r)\cos(\theta)\\
		&+\sum_{\lvert m\rvert>N}a_me^{im\theta}(J_{\lvert m\rvert}^\prime(r)\cos(\theta)-i\sin(\theta)(m/r)J_{\lvert m\rvert}(r))\\
		=&\sum_{j>N}-2^{-j/2}d_jJ_1(r)\cos(\theta)\\
		&+\sum_{\lvert m\rvert>N}\frac{1}{2}a_me^{im\theta}((J_{\lvert m\rvert-1}(r)-J_{\lvert m\rvert+1}(r))\cos(\theta)
		\end{aligned}\\
		-i\sin(\theta)\mathrm{sgn}(m)(J_{\lvert m\rvert-1}(r)+J_{\lvert m\rvert+1}(r))).
		\end{multline*}
		Hence by the triangle inequality and \eqref{e:Bessel bound}, we have 
		\begin{displaymath}
		\mathbb{E} \left(\sup_{x\in B(\beta N)}\lvert \partial_{x_1} f(x)-\partial_{x_1} f_N(x)\rvert^2 \right)\leq c_5 e^{-c_6N}.
		\end{displaymath}
		The proof for $k > 1$ is similar, and we omit the details.
	\end{proof}

	\begin{proof}[Proof of Lemma \ref{l:RPW perturbation}] \textsc{(Assuming Proposition~\ref{p:Gaussian perturbation}.)}
We recall the orthogonal expansion of the Random Plane Wave in \eqref{e:exp} and its finite approximation $f_N$ in \eqref{e:exptrunc}. We define $F:=f_N - \ell$ and $p:=f_N - f$, which are independent $C^\infty$ Gaussian fields. We wish to apply Proposition~\ref{p:Gaussian perturbation} with $\mathcal{D}=D_R$. We first observe that for any $x\in\mathcal{D}_+$, one can easily check that $(p(x),\nabla p(x))$ is non-degenerate using the orthogonal expansion in \eqref{e:exp} (note that by isotropy, it is enough to check non-degeneracy when $x=(x_1,0)$, which easily follows from independence of the variables $b_n$ and $c_n$). This means that the fourth condition of Proposition~\ref{p:Gaussian perturbation} holds.

The first three conditions of Proposition~\ref{p:Gaussian perturbation} hold if we replace $F$ with $f-\ell$ (the first two conditions follow from Assumption~\ref{a:minimal} and by stationarity the third condition is verified by computing the partial derivatives of order at most four of $\kappa$ at the origin). By Lemma~\ref{l:RPWbound}, the covariance function of $F=f_N-\ell$ and its first two derivatives converges uniformly on $B(\beta N)$ to that of $f-\ell$. This implies that once $N$ is sufficiently large, $F$ satisfies the same non-degeneracy conditions as $f-\ell$ on all $D_R$ such that $R\cdot\mathrm{Diam}(D)\leq\beta N$.

We have therefore verified all the conditions of Proposition~\ref{p:Gaussian perturbation} and conclude that with probability one,
\begin{displaymath}
\left\lvert N_\mathrm{ES}\left(F;D_R,0\right)-N_\mathrm{ES}\left(F-p;D_R,0\right)\right\rvert\leq N_\mathrm{QC}\left(F;p,\overline{D_R}\right).
\end{displaymath}

To complete the proof of the lemma, it suffices to show that
\begin{equation}
\label{e:qc3}
\mathbb{E} ( N_\mathrm{QC}(F;p,D_R) )  < c_1  \quad \text{and}  \quad  \mathbb{E} ( N_\mathrm{QC}(F;p,\partial D_R) ) < c_1 
\end{equation}
uniformly over $N$ sufficiently large and $R\cdot\mathrm{Diam}(D)\leq \beta N$. Arguing as in the proof of Lemma \ref{l:Spectral density perturbation} (with one additional application of H\"older's inequality because $p$ is now random), the first bound in \eqref{e:qc3} follows from the Kac-Rice formula once we check the following two conditions. First, for all $N$ sufficiently large
\begin{equation}\label{e:qc6}
\sup_{x\in B(\beta N)}\mathbb{E}\left(\lvert p(x)\rvert^2\middle|G(x,\alpha)=0\right)^{1/2}\leq \frac{1}{N^2},
\end{equation}
where $G(x, \alpha)$ is defined as in \eqref{e:g} with $p_{G(x, \alpha)}$ its density. Second, we require that for all multi-indices $k$ with $\lvert k\rvert=2$
\begin{equation}
\label{e:qc4}
\begin{aligned}
\mathbb{E}& \left(\left\lvert \partial^k F(x) \right\rvert^4 \middle| G(x,\alpha) = 0 \right) p_{G(x, \alpha)}(0)  
\\
& \qquad\qquad\qquad\text{and}  
\\
\mathbb{E} &\left(\left\lvert\partial^k p(x)\right\rvert^4 \middle| G(x,\alpha) = 0 \right) p_{G(x, \alpha)}(0)   
\end{aligned}
\end{equation}
are bounded above uniformly over $(x, \alpha) \in B(\beta N) \times [0,1]$ for sufficiently large $N$.

Since $(p,G)$ is jointly Gaussian, the distribution of $p$ conditional on $G$ is also Gaussian. Therefore in order to prove \eqref{e:qc6}, it is enough to show that
\begin{align*}
\sup_{x\in B(\beta N)}\max\left\{\mathrm{Var}\left(p(x)\middle|G(x,\alpha)=0\right),\mathbb{E}\left(p(x)\middle|G(x,\alpha)=0\right)^2\right\}\leq \frac{1}{N^4}.
\end{align*}
Conditioning on some elements of a Gaussian vector can only reduce the variance of other elements, therefore
\begin{displaymath}
\mathrm{Var}\left(p(x)\middle|G(x,\alpha)=0\right)\leq\mathrm{Var}(p(x))
\end{displaymath}
and the supremum of this quantity over $x\in B(\beta N)$ decays exponentially in $N$ by Lemma~\ref{l:RPWbound}. Turning to the expectation; by Gaussian regression
\begin{displaymath}
\mathbb{E}(p(x)|G(x,\alpha)=0)=
\mathrm{Cov}\left(p(x),\begin{pmatrix}
\nabla(f_N-\alpha p)(x)\\(f_N-\alpha p)(x)
\end{pmatrix}\right)\mathrm{Var}\begin{pmatrix}
\nabla(f_N-\alpha p)(x)\\(f_N-\alpha p)(x)
\end{pmatrix}^{-1}\begin{pmatrix}
0\\0\\\ell
\end{pmatrix}.
\end{displaymath}
By Lemma~\ref{l:RPWbound}, as $N\to\infty$ the covariance of $(\nabla(f_N-\alpha p)(x),(f_N-\alpha p)(x))$ converges uniformly to that of $(\nabla f(x),f(x))$ (which is constant, since $f$ is stationary). Therefore by the Cauchy-Schwarz inequality, the above expression is bounded above, for all $N$ sufficiently large, by $c\mathbb{E}(p(x)^2)^{1/2}$ for some constant $c$ depending only on $\ell$. Applying Lemma~\ref{l:RPWbound} then completes the proof of \eqref{e:qc6}.

A near identical argument shows that the conditional expectations in each term of \eqref{e:qc4} are uniformly bounded for $N$ sufficiently large. It therefore remains to show that $p_{G(x,\alpha)}(0)$ is uniformly bounded. Since this is a Gaussian density, we need only show that the determinant of $\mathrm{Var}(G(x,\alpha))$ is uniformly bounded away from zero. As $N\to\infty$, this covariance matrix converges uniformly to $\mathrm{Var}(\nabla f(x), f(x))$ which has a non-zero determinant that does not depend on $x$. This implies the necessary bound and so completes the proof of the first bound in \eqref{e:qc3}.
		
For the second bound in \eqref{e:qc3}, we count the quasi-critical points on each component of $\partial D_R$ separately and bound their expectation using the same argument as above applied to the boundary.
\end{proof}

\begin{appendix}
\section{Morse theory arguments}\label{a:Appendix}

In this appendix we prove Theorem~\ref{t:Morse}. Throughout this appendix we always assume that $F$, $p$ and $\mathcal{D}$ satisfy Assumption \ref{a:Morse}. The results that we need are very similar to classical results from Morse theory. The main difference is that we work with functions on compact domains with boundary and instead of considering one function at different levels we consider a smooth family of functions at a single level. The first modification is addressed by stratified Morse theory. The second is less studied, but very closely related result have appeared before. The statement that we need is not genuinely new, but we were unable to find a reference which is applicable in exactly our case. The most relevant results are in \cite{handron2002generalized,vakhrameev2000}. 

We start with a very natural statement: as we vary $\alpha$, the set where $g_\alpha:=F-\alpha p\ge 0$ changes continuously (in particular, its topology does not change) unless $\alpha$ passes through a value such that  there is a quasi-critical point at this level. Our proof of this statement is very similar to the standard flow arguments used in Morse theory.

Let us define $A_\alpha=\{x\in \overline{\mathcal{D}}: g_\alpha(x)\ge 0\}$. Our first claim is the following:
\begin{proposition}
\label{prop:Morse}
Let $A_\alpha$ be as above and suppose that $(F,p)$ has no quasi-critical points at level $\alpha$ for $\alpha\in[\alpha_1,\alpha_2]$, then  $A_{\alpha_1}$ and $A_{\alpha_2}$ have the same number of connected components which do not intersect $\partial\mathcal{D}$.
\end{proposition}

\begin{proof}
We will prove that the number of connected components of $A_\alpha$ which do not intersect $\partial\mathcal{D}$ is continuous in $\alpha$, and hence constant. Let us define the flow
\begin{equation}
\label{e:flow}
\frac{d x_\alpha}{d \alpha}=-\partial_\alpha g_\alpha(x_\alpha) \frac{\nabla g_\alpha(x_\alpha)}{|\nabla g_\alpha(x_\alpha)|^2}.
\end{equation}
This defines a smooth flow away from the critical points of $g_\alpha$. An application of the chain rule shows that $g_\alpha(x_\alpha)$ is constant in $\alpha$.

By $\mathcal{D}_\epsilon$ we denote the set of points of $\mathcal{D}$ that are a distance at least $\epsilon$ away from $\partial\mathcal{D}$ and we let $A_{\alpha,\epsilon}=A_\alpha\cap \mathcal{D}_\epsilon$. By `quasi-critical point of $g_\alpha$' we just mean a quasi-critical point of $(F,p)$ at level $\alpha$ (recall that quasi-critical points are defined prior to Assumption~\ref{a:Morse}). We emphasise that the quasi-critical points of $g_\alpha$ are not the same as the critical points of $g_\alpha$; with the latter being defined in the usual way (as points $y$ satisfying $\nabla g_\alpha(y)=0$).

Since $g_\alpha$ has no quasi-critical points we see that all critical points of $g_\alpha$ are separated from its nodal set (the set where $g_\alpha=0$). Furthermore these critical points and the nodal set depend continuously on $\alpha$ (as we describe and justify below). We therefore claim that for every $\alpha_0\in(\alpha_1,\alpha_2)$, we can find $\epsilon,\delta>0$ such that the following hold for all $\alpha \in [\alpha_0-\delta,\alpha_0+\delta]$:
\begin{enumerate}
    \item All components of $\{g_\alpha=0\}\cap \overline{\mathcal{D}}$ are either contained in $\mathcal{D}_{6\epsilon}$, or intersect  $\partial \mathcal{D}_{\epsilon'}$ transversally for all $0\le \epsilon'\le 6\epsilon$ (if the component intersects a corner, this means that the line is transveral to both boundary intervals). All components of $\{g_\alpha=0\}\cap (\overline{\mathcal{D}}\setminus \mathcal{D}_{6\epsilon})$ have exactly one point on $\partial \mathcal{D}$. 
    \item All critical points of $g_\alpha$ are at least distance $2\epsilon$ away from $\{g_\alpha=0\}\cap \mathcal{D}$ (this also includes critical points outside of $\mathcal{D}$). There are no critical points of $g_\alpha$ in  $\mathcal{D}_\epsilon\setminus \mathcal{D}_{6\epsilon}$.  
    \item If we start flow \eqref{e:flow} from a point $x_{\alpha_0}$ which is in an  $\epsilon$-neighbourhood of the boundary of $A_{\alpha_0,3\epsilon}$ then $|x_{\alpha_0}-x_\alpha|<\epsilon$.
\end{enumerate}
Proof of 1. By compactness, there is a small neighbourhood of the boundary such that all components of the restriction of of the  nodal set $\{g_{\alpha_0}=0\}$ to this neighbourhood also intersect $\partial\mathcal{D}$. 
Since there are no quasi-critical points of $g_{\alpha_0}$ on the boundary, all  nodal lines (i.e.\ curves in the nodal set $\{g_{\alpha_0}=0\}$) that intersect $\partial\mathcal{D}$ do so transversally. By transversality, the gradient of $g_{\alpha_0}$ at the point where the nodal line intersects the boundary is not orthogonal to the boundary. By continuity, the same is true in some neighbourhood. This means that the nodal line intersects $\partial\mathcal{D}_{\epsilon'}$ transversally for all sufficiently small $\epsilon'$. Since `sufficiently small' here depends on the second derivatives of $g_{\alpha_0}$, by uniform continuity, we can choose $\epsilon$ such that the above claim holds simultaneously for all $\alpha$ sufficiently close to $\alpha_0$ and all $\epsilon'\le 6\epsilon$.

Proof of 2. It is a standard result that non-degenerate critical points of a smooth family of functions depend continuously on the parameter. Hence, for all $\alpha$ sufficiently close to $\alpha_0$, all critical points of $g_\alpha$ are $\epsilon$-close to the critical points of $g_{\alpha_0}$. Since   there are no critical points of $g_{\alpha_0}$ on the nodal lines, there are no critical points in some neighbourhood of $\{g_\alpha=0\}\cap \mathcal{D}$.  If $g_{\alpha_0}$ has no critical points on $\partial\mathcal{D}$, then by the same argument $g_\alpha$ has no critical points in some neighbourhood of $\partial\mathcal{D}$  for all $\alpha$ sufficiently close to $\alpha_0$. Finally, if $g_{\alpha_0}$ has any critical points on the boundary, then by the same argument as above, for all $\alpha$ sufficiently close to $\alpha_0$ all critical points of $g_\alpha$ are either in a very small neighbourhood of the boundary or outside of a larger neighbourhood of the boundary. 

Proof of 3. The last part follows from the second one. By compactness, $\partial_\alpha g_\alpha$ is uniformly bounded. Since $x_{\alpha_0}$ is away from critical points, we have a uniform lower bound on $|\nabla g_\alpha|$, hence the speed of the flow is uniformly bounded. This means that if $\delta$ is sufficiently small, then the flow can not move by more than $\epsilon$. Note that the flow is not limited to $\mathcal{D}$, it is defined globally outside of critical points. We allow the flow to start in $\mathcal{D}$ and leave it.

From now on we assume that $\alpha_0$, $\epsilon$ and $\delta$ are chosen so that these three statements hold. Since $g_\alpha$ has no critical points at level zero, the sets $\{g_\alpha(x)=0\}$ are made up of $C^2$-smooth curves that either intersect the boundary of $\mathcal{D}$ transversely or do not intersect it at all. This means that $A_\alpha\setminus A_{\alpha,3\epsilon}$ is a disjoint union of quadrilaterals (that is a simply connected domain bounded by a simple piecewise smooth curve with four marked points, these points are `vertices' and the arcs between them are `edges' of the quadrilateral).  For each component of $A_\alpha\cap \partial \mathcal{D}$ there is a quadrilateral such that its four `sides' are: this boundary component, two sub-arcs of $\{g_\alpha=0\}$ and a part of $\partial\mathcal{D}_{3\epsilon}$ (see Figure \ref{fig:flow}). Since each quadrilateral can be retracted to its side, $A_{\alpha,\epsilon}$ is a deformation retract of $A_\alpha$.

\begin{figure}[h]
    \centering
    \includegraphics[width=0.7\textwidth]{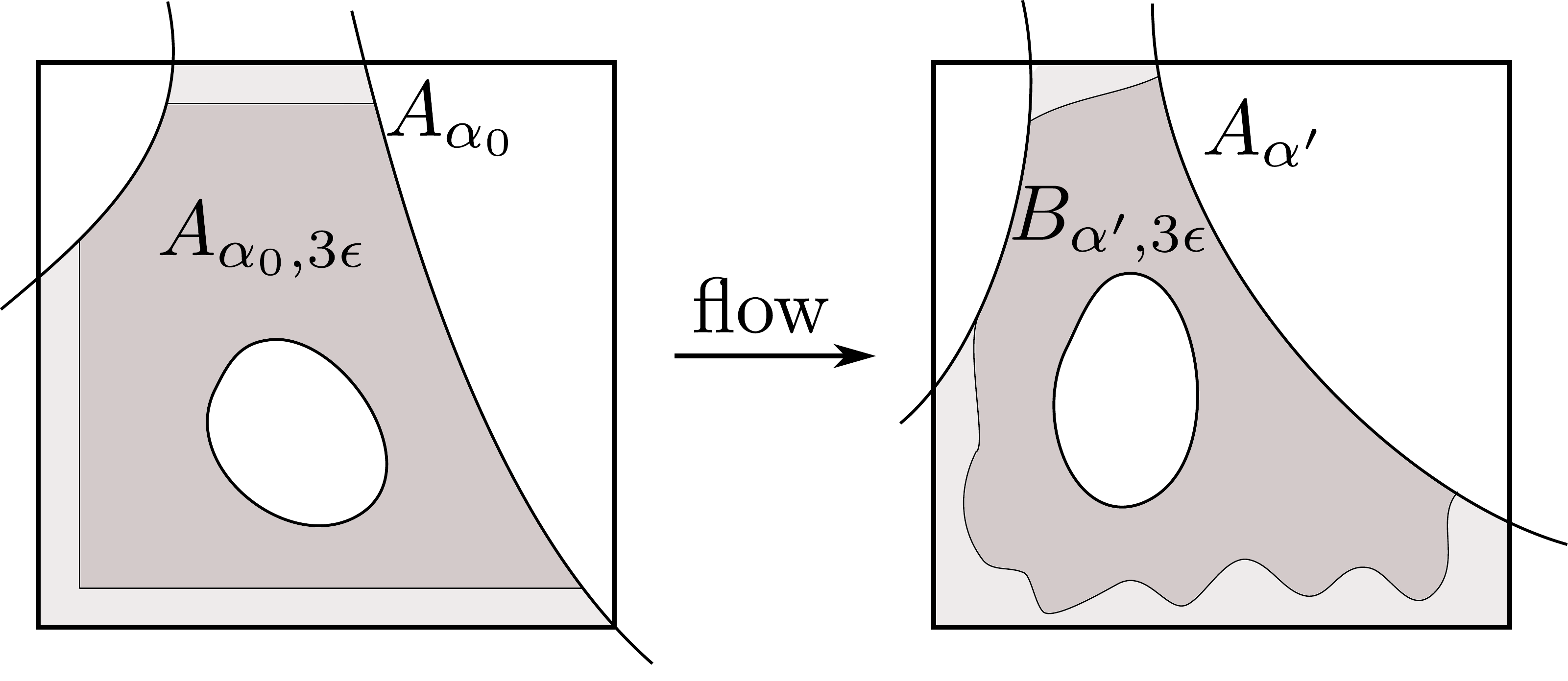}
    \caption{The gradient flow \eqref{e:flow} continuously deforms $\partial A_{\alpha_0,3\epsilon}$ to $\partial B_{\alpha',3\epsilon} $. Clearly $  B_{\alpha',3\epsilon}\subset A_{\alpha'}$. Note that the flow might be not defined \emph{inside} $A_{\alpha',3\epsilon}$, we only move the boundary (although   the flow is really defined and controlled everywhere except small neighbourhoods of critical points, so we can move not only boundary, but also most of the domain by this flow, it is easy to extend the flow in a continuous way to these small neighbourhoods). Light-grey parts are quadrilaterals that can be contracted to the corresponding arcs on the boundary of $A_{\alpha_0,3\epsilon}$ and $B_{\alpha',3\epsilon}$ respectively.  }
    \label{fig:flow}
\end{figure}

Let us consider $\alpha' \in [\alpha_0-\delta,\alpha_0+\delta]$. As $\alpha_0$ changes to $\alpha'$, the flow continuously  moves the boundary of $A_{\alpha_0,3\epsilon}$ by at most $\epsilon$. This gives a homotopy of $\partial A_{\alpha_0,3\epsilon}$ and its image. Since the flow is planar, the image  of $\partial A_{\alpha_0,3\epsilon}$ is the boundary of a domain which is homotopic to $A_{\alpha_0,3\epsilon}$. We denote this domain by $B_{\alpha',3\epsilon}\subset \mathcal{D}$. Since the flow is continuous and preserves $g_\alpha(x_\alpha)$ the following holds: the parts of $\partial A_{\alpha_0,3\epsilon}$ which are made of  nodal lines of $g_{\alpha_0}$ move to nodal lines of $g_{\alpha'}$ and the straight parts of $\partial A_{\alpha_0,3\epsilon}$, where $g_{\alpha_0}\ge 0$, become smooth curves inside $A_{\alpha'}$. Since the flow does not move by more than $\epsilon$, these curves are actually inside $\mathcal{D}_{2\epsilon}$. Moreover, since the flow is invertible,  $A_{\alpha',4\epsilon}\subset B_{\alpha',3\epsilon}\subset A_{\alpha',2\epsilon}$ and  $B_{\alpha',3\epsilon}\cap \mathcal{D}_{4\epsilon}$  is equal to $A_{\alpha',4\epsilon}$. In the same way as before, $A_{\alpha'}\setminus B_{\alpha',3\epsilon}$ is made of quadrilaterals that can be retracted. This proves that $B_{\alpha',3\epsilon}$ is a  deformation retract of $A_{\alpha'}$, so they have the same topology. This proves that $A_{\alpha_0}$ and $A_{\alpha'}$ have the same topology.

Finally we note that in the above argument, the components of $A_{\alpha_0}$ which do not intersect $\partial\mathcal{D}$ are mapped by the flow to components of $A_{\alpha^{\prime}}$ which do not intersect $\partial\mathcal{D}$. Therefore the number of such components is non-decreasing as we move from $\alpha_0$ to $\alpha^{\prime}$. By considering the inverse of flow \eqref{e:flow}, we see that the same is true when moving from $\alpha^{\prime}$ to $\alpha_0$.
This argument proves that the number of connected components which intersect $\partial\mathcal{D}$ is continuous, hence constant.
\end{proof}

Next we would like to analyse what happens when $\alpha$ passes through a critical level (i.e.\ the level of a quasi-critical point). Again, this is very similar to standard Morse theory arguments which give a full CW complex decomposition of excursion sets. Since we are working in the two-dimensional case and only care about the number of connected components, many arguments can be significantly simplified.

We will control the effects of quasi-critical points locally by using a version of the Morse lemma for manifolds with corners. First we require some definitions from \cite{handron2002generalized}. Recall that for $z\in K$ a submanifold with corners, $j(z)$ is the number such that a neighbourhood of $z$ can be mapped by a chart to a neighbourhood of zero in $[0,\infty)^{j(z)}\times\R^{n-j(z)}$. For each $j$, we define the set of $z\in K$ with $j(z)=j$ to be a stratum of $K$. So in the case that $K=\overline{\mathcal{D}}$, the three strata are the interior $\mathcal{D}$, the edges $\mathcal{E}$ and the corners $\mathcal{C}$ (agreeing with our earlier definition of strata for rectangles).
\begin{definition}
Let $M$ be a smooth $n$-dimensional manifold, $K$ be a submanifold with corners and $g:K\to\R$ be a $C^2$ function. We say that $g$ is a Morse function if it satisfies the following:
\begin{enumerate}
    \item If $H$ is a stratum of $K$ then any critical points of $g|_H$ are non-degenerate
    \item If $H_1,H_2$ are (distinct) strata of $K$ with $z\in H_1\subset\overline{H_2}$ then $z$ is not a critical point of $g|_{\overline{H_2}}$.
\end{enumerate}
\end{definition}
\begin{lemma}[Morse lemma, {\cite[Lemma~5]{handron2002generalized}}]\label{l:Morse lemma}
Let $M$ be a smooth $n$-dimensional manifold, $K$ be a submanifold with corners and $g:K\to\R$ be a Morse function. If $z\in K$ is a critical point of $g$ restricted to the stratum of $K$ containing $z$ then there exists $\lambda\in\{0,1,\dots,n-j(z)\}$ and a chart $(U,(x_1,\dots,x_n))$ with $z\in U$ such that 
\[
\begin{aligned}
g=&g(z)-x_1^2-\dots-x_\lambda^2+x_{\lambda+1}^2+\dots+x_{n-j(z)}^2\\
&+(-1)^{\sigma(n-j(z)+1)}x_{n-j(z)+1}+\dots+(-1)^{\sigma(n)}x_n
\end{aligned}
\]
holds in $U\cap K$ for some $\sigma(n-j(z)+1),\dots,\sigma(n)\in \mathbb{N}$. We define this number $\lambda$ to be the index of the critical point $z$.
\end{lemma}
We note that \cite[Lemma~5]{handron2002generalized} is stated slightly differently (with a fixed sign in front of the terms $x_{n-j(z)+1},\dots,x_n$) because it is stated under a certain assumption on the signs of components of $\nabla g$.

We also comment that we only use this lemma in the case $n=2$, however for notational convenience we will split the local variables $x_i$ into three groups as above. The convention of course, will be that some of these $x_i$ are identically zero.

If $F$ and $p$ satisfy Assumption~\ref{a:Morse}, then for any $x\in\overline{\mathcal{D}}$ such that $p(x)\neq 0$, there is a neighbourhood of $x$ on which $F/p$ is the restriction of a Morse function on $\overline{\mathcal{D}}$. This follows easily from considering the gradient and Hessian of $F/p$.

\begin{proposition}\label{t:Morse single level}
Let $\mathcal{D}\subset\R^2$ be an open rectangle and let $(F,p)$ satisfy Assumption~\ref{a:Morse}. Suppose that $(F,p)$ has a quasi-critical point at level $\beta$, then for $\delta>0$ sufficiently small, the number of connected components of $A_{\beta+\delta}$ which do not intersect $\partial\mathcal{D}$ differs from the corresponding number for $A_{\beta-\delta}$ by at most one (and the corresponding number for $A_\beta$ is between these two).
The same conclusion holds for the level sets $\{F-\alpha p=0\}\cap\overline{\mathcal{D}}$ for $\alpha\in\{\beta-\delta,\beta+\delta,\beta\}$.
\end{proposition}

This result is almost a corollary of  \cite[Theorem~8]{handron2002generalized} which gives a CW-decomposition of excursion sets, but our setting is slightly different, so their result can not be applied as stated in our situation. It is possible to modify the proof of \cite[Theorem~8]{handron2002generalized} in order to prove Proposition~\ref{t:Morse single level}. We do not follow this road, instead we give a different proof based on Lemma \ref{l:Morse lemma} and the flow argument used in Proposition \ref{prop:Morse}.

\begin{proof}
By Proposition~\ref{prop:Morse}, it is enough to prove the result for some $\delta>0$ such that there are no quasi-critical points with level in $[\beta-\delta,\beta+\delta]\backslash\{\beta\}$. Let $z$ denote the quasi-critical point of $(F,p)$ at level $\beta$ and choose $\delta>0$ such that this is the only quasi-critical point with level in $[\beta-\delta,\beta+\delta]$ (this is possible because $\overline{\mathcal{D}}$ is compact and quasi-critical points are isolated). Since $p(z)\neq 0$ by Assumption~\ref{a:Morse}, we can find a neighbourhood $U$ of $z$ on which $p\neq 0$, and hence $F/p$ is well defined on this neighbourhood. We will assume $p(z)>0$; if this value was negative our arguments could be repeated with some sign changes. Since
\begin{displaymath}
\nabla(F/p)=\frac{1}{p^2}(p\nabla F-F\nabla p)=\frac{1}{p}(\nabla F-(F/p)\nabla p)
\end{displaymath}
it is clear that $F-\alpha p$ has a quasi-critical point at $x$ where $p(x)\neq 0$ if and only if $x$ is a critical point of $F/p$ restricted to the stratum of $x$ at level $\alpha$. It is also obvious that $F/p-\alpha\ge 0$ if and only if $F-\alpha p\ge 0$, for points at which $p>0$. By Lemma~\ref{l:Morse lemma} we have in local coordinates on some possibly smaller neighbourhood $U$ of $z$
\begin{equation}
\label{e:local coord}
\begin{aligned}
F/p=&\beta-x_1^2-\dots-x_\lambda^2+x_{\lambda+1}^2+\dots+x_{n-j(z)}^2\\
&+(-1)^{\sigma(n-j(z)+1)}x_{n-j(z)+1}+\dots+(-1)^{\sigma(n)}x_n.
\end{aligned}
\end{equation}

By further decreasing $U$, we can assume the following: if there are zero-level lines of $F-\beta p$ emerging from $z$ then $U$ is bounded either by a smooth curve which crosses the zero-level lines of $F-\beta p$ or, if $z\in \partial \mathcal{D}$, by the union of such a curve and a sub-arc of $\partial \mathcal{D}$. If there are no zero-level lines of $F-\beta p$ emerging from $z$, then the zero set of $F-\beta p$ in a $\overline{\mathcal{D}}$-neighbourhood of $z$ is just $\{z\}$. We can then choose $U$ to be such a neighbourhood, and we further assume that the boundary of $U$ is either a circle or the union of an arc of a circle and a sub-arc of $\partial\mathcal{D}$. Let $B$ be an even smaller neighbourhood of $z$ such that $\overline{B}\subset U$.

First we consider the case that there are zero-level lines of $F-\beta p$ emerging from $z$. Arguing in the same way as in the proof of Proposition~\ref{prop:Morse} we can choose $U$ and $B$ in such a way that all components of $\{F-\beta p\ge 0\}\cap (U\setminus B)$ are quadrilaterals such that their boundaries are made of two sub-arcs of $\partial B$ and $\partial U$ and two arcs that are either sub-arcs of $\{F-\beta p=0\}$ or $\partial \mathcal{D}$. The flow \eqref{e:flow} is uniformly bounded in a small neighbourhood of $\{F-\beta p=0\}\setminus B$ and there are are no critical points in this neighbourhood. As before, this implies that there is a sufficiently small $\delta>0$ such that if we run the flow starting from $\alpha=\beta$ and ending inside $[\beta-\delta,\beta+\delta]$ then the curves $\{F-\beta p=0\}\setminus B$ move by less than $\mathrm{dist}(B,\mathcal{D}\setminus U)$. This means that the topology of the excursion sets $A_\alpha$ and nodal sets $\{g_\alpha=0\}$ outside of $U$ does not change (as we pass from $\alpha=\beta-\delta$ to $\alpha=\beta+\delta$) and the number of arcs where the nodal set intersects $\partial U$ does not change as well. This proves that any change in the number of excursion set components which do not intersect $\partial\mathcal{D}$ must happen inside of $U$.

In the case that there are no zero-level lines of $F-\beta p$ emerging from $z$, we know that $F-\beta p$ is bounded away from zero on $U\setminus B$ and there are no critical points on this set. Therefore in this case we can also choose $\delta>0$ such that $F-\alpha p$ is non-zero on $U\setminus B$ for all $\alpha\in[\beta-\delta,\beta+\delta]$. Once again we conclude that any change in the topology of $A_\alpha$ for this range of $\alpha$ must occur inside $U$.

Inside of $U$, we consider the excursion sets $\{F/p\geq\alpha\}$ rather than $A_\alpha=\{F-\alpha p\geq 0\}$ because we can describe the former using the simple coordinates in \eqref{e:local coord}. So depending on the dimension of the stratum containing $z$ and the index of $z$, we have the following options for the \emph{local behaviour} in $U$ when $\alpha$ is close to $\beta$: 

\begin{enumerate}
    \item Dimension $2$ ($z\in\mathcal{D}$): If the index $\lambda$ is $0$ or $2$, then the nodal lines of $F-\alpha p$ do not intersect $\partial U$. As $\alpha$ moves from below $\beta$ to above $\beta$, the component of $\{F/p-\alpha\geq 0\}$ changes from a small disc, to a point $z$, to the empty set or the other way round. If the index is $1$, then the excursion component in $U$ changes from a single component bounded by a hyperbola, to two components separated by a hyperbola (or the other way round).
    \item Dimension $1$ ($z\in\mathcal{E}$): At $\alpha=\beta$ the excursion component is either above a parabola, in which case the topology does not change, or it is the domain between a sub-arc of $\partial\mathcal{D}$ and a parabola. In the latter case, the excursion set in $U$ can change from one component to two components (or the other way round). See Figure~\ref{fig:qc} for an example.
    \item Dimension $0$ ($z\in\mathcal{C}$): In this case, if both $\sigma(1)$ and $\sigma(2)$ are the same, then $\{z\}$ is an isolated component of the nodal set of $F-\beta p$ and it can either disappear or become an interval. So the topology of the excursion set either stays the same or we add one component. If $\sigma(1)\ne \sigma(2)$, then  the nodal line of $F-\beta p$ is a line through $z$. When $\alpha$ changes, this line shifts, but the number of excursion set components does not change.  
\end{enumerate}
\begin{figure}[h]
    \centering
    \includegraphics[width=0.8\textwidth]{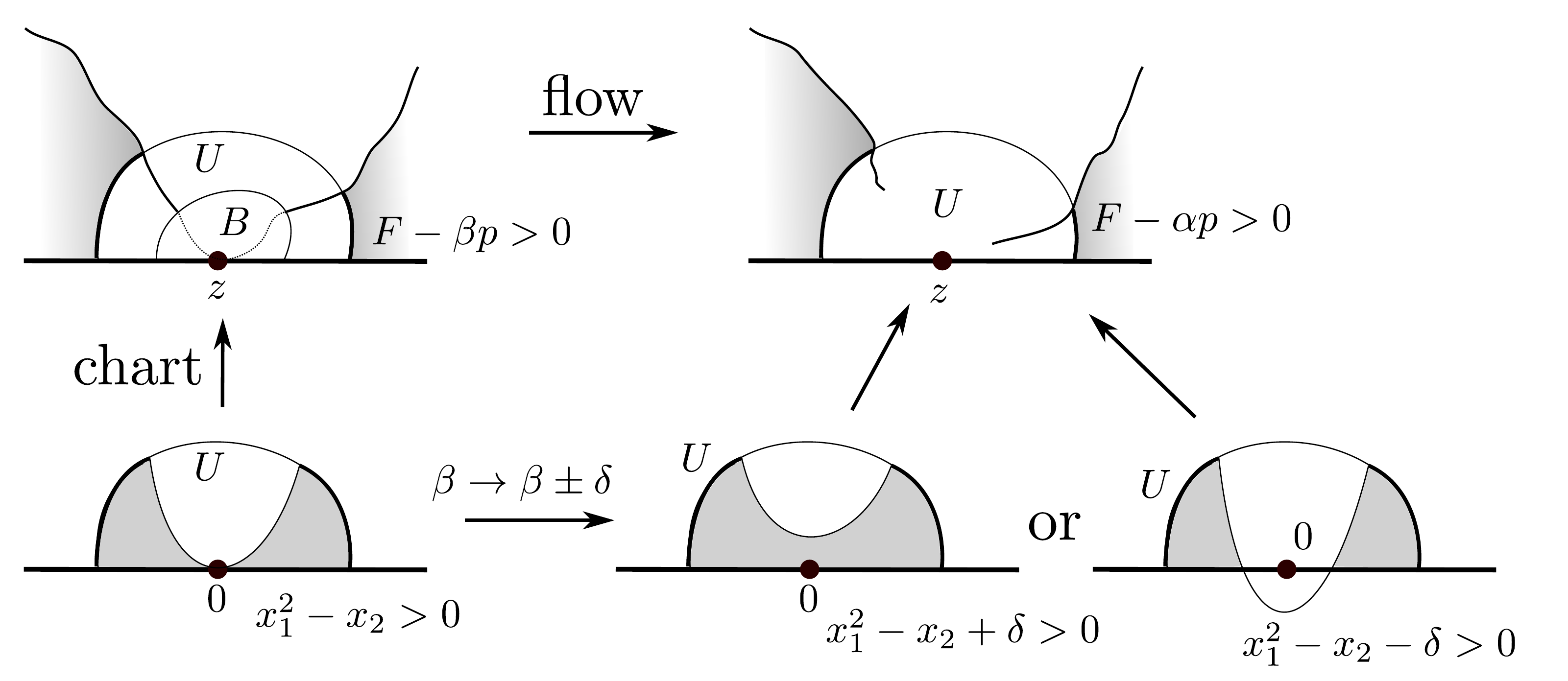}
    \caption{Change of the excursion set near a quasi critical point of index $1$ on a one-dimensional stratum. The analysis of other cases is very similar}
    \label{fig:qc}
\end{figure}
Summing this up, we see that inside $U$ the number of components of $A_\alpha$ which do not intersect $\partial\mathcal{D}$ can not change by more than one, and outside $U$ the number of such domains does not change at all. The arguments above also show that the same is true of the number of nodal components, and so this completes the proof.
\end{proof}
\begin{proof}[Proof of Theorem~\ref{t:Morse}]
This follows immediately from applying Proposition~\ref{t:Morse single level} to each quasi-critical point of $(F,p)$ and Proposition~\ref{prop:Morse} at levels without such points.
\end{proof}

\end{appendix}

\bigskip
\printbibliography

\end{document}